\documentclass[a4paper,10pt]{amsart}
\usepackage{amsmath,amsthm,amssymb,latexsym,enumerate,xcolor,hyperref}
\usepackage{graphicx}
\usepackage[margin=2.5cm]{geometry}

\numberwithin{equation}{section}

\begin{document}

	\newtheorem{thm}{Theorem}[section]
	\newtheorem{prop}[thm]{Proposition}
	\newtheorem{lem}[thm]{Lemma}
	\newtheorem{cor}[thm]{Corollary}
	\newtheorem{rem}[thm]{Remark}
	\newtheorem*{defn}{Definition}

	\newtheorem{definit}[thm]{Definition}
	\newtheorem{setting}{Setting}
	\renewcommand{\thesetting}{\Alph{setting}}
	
	\newcommand{\DD}{\mathbb{D}}
	\newcommand{\NN}{\mathbb{N}}
	\newcommand{\ZZ}{\mathbb{Z}}
	\newcommand{\QQ}{\mathbb{Q}}
	\newcommand{\RR}{\mathbb{R}}
	\newcommand{\CC}{\mathbb{C}}
	\renewcommand{\SS}{\mathbb{S}}

	\renewcommand{\theequation}{\arabic{section}.\arabic{equation}}

	\newcommand{\supp}{\mathop{\mathrm{supp}}}    
	
	\newcommand{\re}{\mathop{\mathrm{Re}}}   
	\newcommand{\im}{\mathop{\mathrm{Im}}}   
	\newcommand{\dist}{\mathop{\mathrm{dist}}}  
	\newcommand{\link}{\mathop{\circ\kern-.35em -}}
	\newcommand{\spn}{\mathop{\mathrm{span}}}   
	\newcommand{\ind}{\mathop{\mathrm{ind}}}   
	\newcommand{\rank}{\mathop{\mathrm{rank}}}   
	\newcommand{\Fix}{\mathop{\mathrm{Fix}}}   
	\newcommand{\codim}{\mathop{\mathrm{codim}}}   
	\newcommand{\conv}{\mathop{\mathrm{conv}}}   
	\newcommand{\epsi}{\mbox{$\varepsilon$}}
	\newcommand{\eps}{\mathchoice{\epsi}{\epsi}
		{\mbox{\scriptsize\epsi}}{\mbox{\tiny\epsi}}}
	\newcommand{\cl}{\overline}
	\newcommand{\pa}{\partial}
	\newcommand{\ve}{\varepsilon}
	\newcommand{\zi}{\zeta}
	\newcommand{\Si}{\Sigma}
	\newcommand{\cA}{{\mathcal A}}
	\newcommand{\cG}{{\mathcal G}}
	\newcommand{\cH}{{\mathcal H}}
	\newcommand{\cI}{{\mathcal I}}
	\newcommand{\cJ}{{\mathcal J}}
	\newcommand{\cK}{{\mathcal K}}
	\newcommand{\cL}{{\mathcal L}}
	\newcommand{\cN}{{\mathcal N}}
	\newcommand{\cR}{{\mathcal R}}
	\newcommand{\cS}{{\mathcal S}}
	\newcommand{\cT}{{\mathcal T}}
	\newcommand{\cU}{{\mathcal U}}
	\newcommand{\OM}{\Omega}
	\newcommand{\B}{\bullet}
	\newcommand{\ol}{\overline}
	\newcommand{\ul}{\underline}
	\newcommand{\vp}{\varphi}
	\newcommand{\AC}{\mathop{\mathrm{AC}}}   
	\newcommand{\Lip}{\mathop{\mathrm{Lip}}}   
	\newcommand{\es}{\mathop{\mathrm{esssup}}}   
	\newcommand{\les}{\mathop{\mathrm{les}}}   
	\newcommand{\nid}{\noindent}
	\newcommand{\pzr}{\phi^0_R}
	\newcommand{\pir}{\phi^\infty_R}
	\newcommand{\psr}{\phi^*_R}
	\newcommand{\pow}{\frac{N}{N-1}}
	\newcommand{\ncl}{\mathop{\mathrm{nc-lim}}}   
	\newcommand{\nvl}{\mathop{\mathrm{nv-lim}}}  
	\newcommand{\la}{\lambda}
	\newcommand{\La}{\Lambda}    
	\newcommand{\de}{\delta}    
	\newcommand{\fhi}{\varphi} 
	\newcommand{\ga}{\gamma}    
	\newcommand{\ka}{\kappa}   
	
	\newcommand{\core}{\heartsuit}
	\newcommand{\diam}{\mathrm{diam}}

	\newcommand{\lan}{\langle}
	\newcommand{\ran}{\rangle}
	\newcommand{\tr}{\mathop{\mathrm{tr}}}
	\newcommand{\diag}{\mathop{\mathrm{diag}}}
	\newcommand{\dv}{\mathop{\mathrm{div}}}
	
	\newcommand{\al}{\alpha}
	\newcommand{\be}{\beta}
	\newcommand{\Om}{\Omega}
	\newcommand{\na}{\nabla}
	
	\newcommand{\cC}{\mathcal{C}}
	\newcommand{\cM}{\mathcal{M}}
	\newcommand{\nr}{\Vert}
	\newcommand{\De}{\Delta}
	\newcommand{\cX}{\mathcal{X}}
	\newcommand{\cP}{\mathcal{P}}
	\newcommand{\om}{\omega}
	\newcommand{\si}{\sigma}
	\newcommand{\te}{\theta}
	\newcommand{\Ga}{\Gamma}
	
	\newcommand{\vV}{\mathbf{v}}
	\newcommand{\lbunu}{\ul{m}}
	\newcommand{\ca}{\ul{a}}
	\newcommand{\Vve}{\ul{\varepsilon}}
	\newcommand{\ur}{\ul{r}}
	
	\title[Soap bubbles and optimal quantitative stability]{
		Optimal quantitative stability estimates for Alexandrov's Soap Bubble Theorem via Gagliardo-Nirenberg-type interpolation inequalities}
	

	\author{João Gonçalves da Silva}
	\address{Department of Mathematics and Statistics, The University of Western Australia, Crawley, Perth, WA 6009, Australia}
	\email{joao.goncalvesdasilva@research.uwa.edu.au}

	\author{Giorgio Poggesi}
	\address{Discipline of Mathematical Sciences, The University of Adelaide, Adelaide SA 5005, Australia.}
	%
	%
	\email{giorgio.poggesi@adelaide.edu.au}

	\date{\today}
	

	\begin{abstract}
		The paper provides optimal quantitative stability estimates for the celebrated Alexandrov's Soap Bubble Theorem within the class of $C^{k,\al}$ domains, for any $k \ge 1$ and $0 < \al \le 1$, by leveraging Gagliardo-Nirenberg-type interpolation inequalities. 
        Optimal estimates of uniform closeness to a ball are established for $L^r$ deviations of the mean curvature from being constant, for any $r\ge 2$ (more generally, for any $r>1$ such that $r\ge (2N-2)/(N+1)$).

        For $r>\frac{N-1}{2}$, the stability profile is linear, thus returning the existing results established in the literature through computations for nearly spherical sets. All the stability estimates for $r\le \frac{N-1}{2}$, for which the profile is not linear, are new; even in the particular case $r=2$ (which has been extensively studied, since it is a case of interest for several critical applications), the sharp stability profile that we obtain is new. Interestingly, we also prove that the (non-linear) profile for $r \le \frac{N-1}{2}$ improves as $k$ becomes larger to such an extent that it becomes formally linear
        as $k$ goes to $\infty$.
        
        Finally,
        for any $k \ge 1$ and $0< \al \le 1$, we show that our estimates are optimal within the class of $C^{k,\al}$ domains, by providing explicit examples. 
        %
	\end{abstract}


	\maketitle

	\raggedbottom

	\section{Introduction}
	
	\subsection{State of the art and motivation}
	We are interested in quantitative stability for the celebrated Alexandrov's Soap Bubble Theorem. Such a celebrated rigidity result deals with the mean curvature $H$ of the boundary $\Ga$ of a bounded domain $\Om \subset \RR^N$. 
	General quantitative results dealing with bubbling phenomena have been obtained in \cite{CirMag} (for uniform deviation of the mean curvature from being constant, i.e., $\nr H - H_0 \nr_{L^\infty(\Ga)}$, where $H_0$ is some reference constant), in \cite{JN} (for $L^{N-1}$-type deviations, i.e., $\nr H - H_0 \nr_{L^{N-1}(\Ga)}$), and finally in \cite{Pogbubbling} (for $L^{1}$-type deviations\footnote{More in general, for deviations such as $\int_{\Ga} \left( H - H_0 \right)^+ dS_x$.} in any dimension).
	Under assumptions preventing bubbling phenomena (e.g., a uniform sphere condition or isoperimetric-type bounds), various sharp quantitative stability results of proximity to a single ball have been established. In \cite{CV} a sharp stability estimate (of Hausdorff-type closeness) was obtained for the uniform deviation $\nr H - H_0 \nr_{L^\infty(\Ga)}$. Weaker deviations were considered in \cite{KM,MP,MP2,MP3,MP6,FZ}. In particular, for $L^2$-deviations, \cite[Theorem 4.6]{MP2} gives the following
    (sharp) 
    linear stability estimate:
	\begin{equation}\label{eq:MP L2}
		\frac{|\Om \De B_{1/H_0}(z)|}{|B_{1/H_0} (z)|}
        \le C \, \nr H - H_0 \nr_{L^2(\Ga)} ,
	\end{equation}
	where $z$ is a point in $\Om$, $H_0:=\frac{|\Ga|}{N|\Om|}$, and $B_{1/H_0}(z)$ denotes the ball centered at $z$ with radius $1 / H_0$. As shown in \cite{Pog}, the asymmetry in measure in \eqref{eq:MP L2} can be improved and replaced with a stronger $L^2$-type distance, that is,
    \begin{equation*}
		\left\nr |x-z| - \frac{1}{H_0} \right\nr_{L^2 (\Ga)} \le C \, \nr H - H_0 \nr_{L^2(\Ga)} .
	\end{equation*}
    Moreover, as shown in \cite[Theorem 3.9, Formula (3.10)]{MP6}, linear stability can be obtained even for the (stronger) $L^2$-distance of the Gauss map, that is, the following (sharp) estimate holds true:
	\begin{equation}\label{eq:MP L2 + Pogcones}
		\left\nr |x-z| - \frac{1}{H_0} \right\nr_{L^2 (\Ga)} + \left\nr \frac{1}{H_0} \, \nu - ( x-z ) \right\nr_{L^2(\Ga)} \le C \, \nr H - H_0 \nr_{L^2(\Ga)} ,
	\end{equation}
    where $\nu$ denotes the exterior unit normal to $\Ga$.
    
	If on the left-hand side, we want to measure the closeness to a ball with a uniform norm, then, at least in high dimensions, we have to either pay something in the stability profile\footnote{As we are going to show, in this case we cannot expect to obtain a linear profile for $N \ge 5$.} or strengthen the $L^2$ measure of the deviation on the right-hand side. On the one hand, \cite[Theorem 3.9]{MP6} established the following stability result for uniformly $C^{2,\al}$ domains:
	\begin{equation}\label{eq: MP MinE result}
		\rho_e(z) - \rho_i(z) \le C \, 
		\begin{cases}
			\nr H - H_0 \nr_{L^2(\Ga)} \quad & \text{ for } N=2,3 ,
			\\
			\nr H - H_0 \nr_{L^2(\Ga)} \max \left[ \log\left( \frac{1}{\nr H - H_0 \nr_{L^2(\Ga)}} \right) , 1 \right] \quad & \text{ for } N= 4 ,
			\\
			\nr H - H_0 \nr_{L^2(\Ga)}^{4/N} \quad & \text{ for } N\ge 5 ,
		\end{cases}
	\end{equation}
	where $z$ is a point in $\Om$, and we have set 
	$$
	\rho_e(z) := \max_{x\in\Ga} |x-z| \quad \text{and} \quad \rho_i(z) := \max_{x\in\Ga} |x-z|.
	$$
    On the other hand, in light of the computations for nearly spherical sets in \cite[Theorem 1.10]{KM} (see also \cite{FZ}), we know that linear stability estimates for the uniform closeness $\rho_e (z)-\rho_i (z)$ remain valid provided that we use deviations $\nr H - H_0 \nr_{L^r (\Ga)}$ with $r > (N-1)/2$ if $N \ge 5$, and $r\ge2$ if $N\le 4$.
    
	In the present paper, we establish the sharp stability profile concerning the uniform closeness $\rho_e (z)-\rho_i (z)$ and $\nr H - H_0 \nr_{L^r (\Ga)}$, for any $r \ge 2$ (more generally, for any $r > 1$ such that $r \ge (2N-2)/(N+1)$); for any $k\ge 1$ and $0<\al\le 1$, our estimates are established for general uniformly $C^{k,\al}$ domains, without requiring further assumptions.
    The sharp general stability profile is presented in Theorem \ref{thm:final stability theorem} below. 
    For $r>\frac{N-1}{2}$ the stability profile is linear, thus returning the existing results established in the literature through computations for nearly spherical sets (\cite{KM,FZ}). Surprisingly, all the stability estimates for $r\le \frac{N-1}{2}$, for which the profile is not linear, are new. Interestingly, we also prove that the (non-linear) profile for $r \le \frac{N-1}{2}$ improves as $k$ becomes larger to such an extent that it becomes formally linear (in any dimension) as $k$ goes to $\infty$.


    In the present paper, we also show that the stability estimates obtained here are optimal for any $k \ge 1$ and $0< \al \le 1$.
    We stress that, while the optimality of the linear stability estimates (such as, \eqref{eq:MP L2} and \eqref{eq:MP L2 + Pogcones} for any $N$, and \eqref{eq:final stability} for $r > (N-1)/2 $) can be readily checked by a simple calculation with ellipsoids, a careful analysis is needed to check that the stability exponent $\tau_{k,\al,N,r}$ in \eqref{eq:final stability} is optimal, which we provide in the present paper.

    We stress that the sharp profile we establish in Theorem \ref{thm:final stability theorem} is new even in the particular case $r=2$. In fact, in the particular case where one considers
    %
    %
    $r=2$ and $C^{1, 1}$ domains, the general result in
    Theorem~\ref{thm:final stability theorem} reduces to
    \begin{equation*}
		\rho_e(z) - \rho_i(z) \le C \, 
		\begin{cases}
			\nr H - H_0 \nr_{L^2(\Ga)} \quad & \text{ for } N=2,3,4 ,
			\\
			\nr H - H_0 \nr_{L^2(\Ga)} \max \left[ \frac{1}{2}\log\left( \frac{1}{\nr H - H_0 \nr_{L^2(\Ga)}} \right) , 1 \right] \quad & \text{ for } N= 5 ,
			\\
			\nr H - H_0 \nr_{L^2(\Ga)}^{4/(N-1)} \quad & \text{ for } N\ge 6 ,
		\end{cases}
	\end{equation*}
    which clearly improves the stability profile in \eqref{eq: MP MinE result} for any $N \ge 4$.
   A key observation that allows such an improvement of the stability profile is to notice that if the deviation $\nr H-H_0\nr_{L^r(\Ga)}$ is small enough -- and we are working within a class of uniformly $C^{k,\al}$ ($k\geq 1$ and $0< \al \le 1$) domains to prevent bubbling phenomena --, then $\Om$ must be nearly spherical. A similar remark was recently noticed in \cite{FZ} by leveraging the elliptic regularity theory for almost minimal hypersurfaces; here, we avoid exploiting the elliptic regularity theory, but instead will explicitly deduce this qualitative information from \eqref{eq:MP L2 + Pogcones}. Such a remark allows to reduce the dimension from $N$ to $N-1$, with consequent improvement of the profile when applying Sobolev embedding and interpolation inequalities.


	\subsection{Main results: optimal quantitative stability estimates for Alexandrov's Soap Bubble Theorem with $L^r$-type deviations}

	\begin{thm}\label{thm:SBT stability}
    Let $N\geq 4$ be an integer, $r\in \left[\frac{2N-2}{N+1},\frac{N-1}{2}\right)$, and $\Omega\subset \RR^{N}$ a bounded domain with boundary $\Ga$ of class $C^{k,\alpha}$, where $k\geq 1$ and $0<\alpha \leq 1$. If $k =1$, then we further assume that $\Ga$ is of class $W^{2,r}$.
    Set $H_0:=|\Ga|/(N | \Om|)$. 
		
		Then, there exists a point $z\in \Om$ such that
		\begin{equation}\label{eq:def rhoe rhoi}
			\rho_e(z) := \max_{x\in\Ga} |x-z| \quad \text{and} \quad \rho_i(z) := \max_{x\in\Ga} |x-z|
		\end{equation}
		satisfy
        \begin{equation}\label{eq:stability SBTGN}
			\rho_e(z) -\rho_i(z) \le C \nr H_0 - H \nr_{L^r(\Ga)}^{\tau_{k,\alpha,N,r}} ,
		\end{equation}
		where
		\begin{equation}\label{eq:stability exponent I}
			\tau_{k,\alpha,N,r} := \frac{k+\alpha}{k+\alpha + \frac{N-1-2r}{r}},
		\end{equation}
		and $C$ is a constant only depending on $N$, $k$, $\alpha$, $r$, the $C^{k,\al}$-regularity of $\Ga$ and the diameter of $\Omega$.
	\end{thm}
	\begin{rem}\label{rem:first big remark}
		{\rm 
			\begin{enumerate}[(i)]
				\item Notice that, as $k\to \infty$,~\eqref{eq:stability SBTGN} becomes (formally) linear
				$$\lim_{k\to\infty} \tau_{k,\alpha,N,r} = 1 .$$
				\item The stability exponent~\eqref{eq:stability exponent I} is continuous with respect to $k+\alpha$, that is, 
                $$
                \lim_{\alpha\to 0^+}\tau_{k+1,\alpha,N,r}= \tau_{k,1,N,r}
                $$
			\end{enumerate}
		}
	\end{rem}
	
	The next theorem deals with the limit case $r=\frac{N-1}{2}$.
	\begin{thm}[The case $r=\frac{N-1}{2}$]
		\label{thm:SBT log stability case r = (N-1)/2} 
        Let $N\geq 4$, $k \ge 1$ an integer. Let $\Om\subset \RR^N$ be a bounded domain with boundary $\Ga$ of class $C^{k,\al}$, with $0<\al\leq 1$, and set $H_0:=|\Ga|/(N | \Om|)$.
		If $k=1$, further assume that $\Ga$ is of class~$W^{2,\frac{N-1}{2}}$.
		
		Then, there exists a point $z\in \Om$ such that the radii $\rho_e(z)$ and $\rho_i(z)$ defined in \eqref{eq:def rhoe rhoi} satisfy
		\begin{equation}\label{eq:case r=(N-1)/2 log stability SBT via Holder estimates}
			\rho_e(z) -\rho_i(z) \le C \nr H_0 - H \nr_{L^{\frac{N-1}{2}}(\Ga)} \max\left\lbrace  \frac{1}{k+\alpha} \, \log \left(\frac{1}{\nr H_0 - H \nr_{L^{\frac{N-1}{2}}(\Ga)}}\right) , 1 \right\rbrace ,
		\end{equation}
		where $C$ is a constant only depending on $N$, $k$, the $C^{k,\al}$-regularity of $\Ga$ and the diameter of $\Omega$. 
	\end{thm}
	
	\begin{rem}
		{\rm
				 Notice that in the case $r=\frac{N-1}{2}$, we still have that \eqref{eq:case r=(N-1)/2 log stability SBT via Holder estimates} becomes (formally) linear as $k\to\infty$, being as
				\begin{equation*}
					\lim_{k\to\infty} \max\left\lbrace  \frac{1}{k+\alpha} \, \log \left(\frac{1}{\nr H_0 - H \nr_{L^{\frac{N-1}{2}}(\Ga)}}\right) , 1 \right\rbrace = 1.
				\end{equation*}
		} 
	\end{rem}	

        
    With the two previous results, we are now able to present a complete stability profile. This is summarised in the following theorem.

        \begin{thm}\label{thm:final stability theorem}
           Let $N\geq 2$ and $k\geq 1$ be two integers. Let $\Omega\subset\mathbb{R}^N$ be a bounded domain with boundary $\Ga$ of class $C^{k,\alpha}$, with $0< \alpha\leq 1$, and set $H_0=\frac{|\Ga|}{N|\Omega|}$. Let $r>1$ be such that $r\geq \frac{2N-2}{N+1}$. If $k=1$, we further assume that $\Ga$ is of class $W^{2,r}$.
           
           Then, there exists a point $z\in \Omega$ such that the radii $\rho_e(z)$ and $\rho_i(z)$ defined in \eqref{eq:def rhoe rhoi} satisfy
           \begin{equation}\label{eq:final stability}
               \rho_e(z)-\rho_i(z) \leq C\begin{cases}
                    \nr H - H_0 \nr_{L^r(\Ga)} \quad & \text{ if } r>\frac{N-1}{2} \\
                    \nr H_0 - H \nr_{L^{r}(\Ga)} \max\left\lbrace  \frac{1}{k+\alpha} \, \log \left(\frac{1}{\nr H_0 - H \nr_{L^{r}(\Ga)}}\right) , 1 \right\rbrace \quad & \text{ if } r = \frac{N-1}{2}\\
                    \nr H_0 - H \nr_{L^r(\Ga)}^{\tau_{k,\alpha,N,r}} \quad & \text{ if } r < \frac{N-1}{2},
               \end{cases}
           \end{equation}
           where $C$ is a constant only depending on $N$, $k$, $\alpha$, $r$, the diameter of $\Om$ and the $C^{k,\al}$-regularity of $\Ga$, and $\tau_{k,\alpha,N,r}$ is defined in \eqref{eq:stability exponent I}.
        \end{thm}
        \begin{rem}
            \rm{\begin{enumerate}[(i)]
                \item The sharpness of the exponent $1$ in the linear estimates \big(i.e., for $r> \frac{N-1}{2}$\big) can be readily checked by a simple calculation with ellipsoids.
                \item For $k\geq 1$ and $0<\al\le 1$, the stability exponent $\tau_{k,\al,N,r}$ is sharp, as we prove in Section \ref{Sec: Optimality} and Appendix \ref{sec:Optimality in the rough case}.
            \end{enumerate}}
        \end{rem}
        
        As already mentioned, all the stability estimates for $r\le \frac{N-1}{2}$, for which the profile is not linear, are new (even in the particular case where one considers $L^2$ deviations of the mean curvature from being constant).
        The linear estimates (i.e., the case $r>\frac{N-1}{2}$) essentially return
        those in \cite{KM,FZ}. We also point out that in the case $r = \infty$, our stability estimates formally recover the sharp linear estimate in \cite{CV}; more precisely, while \cite{CV} provides a linear estimate where the constant depends on $N$, $|\Ga|$, and the $C^{1,1}$ regularity\footnote{More precisely, the regularity parameter used in \cite{CV} is the radius of the uniform interior and exterior sphere conditions, which is equivalent to the $C^{1,1}$ regularity of $\Ga$ (see, e.g., \cite[Corollary 3.14]{ABMMZ}).}, our result shows that a linear estimate remains valid even if we replace the dependence on the  $C^{1,1}$ regularity with the dependence on the $C^{1,\al}$ regularity.

	\subsection{Organization of the paper}  
	The rest of the paper is organized as follows. 
	
	In Section \ref{sec:Preliminaries} we introduce the setting and recall some basic tools that will be used in the proof of our main results.
	
	Section \ref{sec:Proof of Main Theorems} is devoted to the proofs of Theorems \ref{thm:SBT stability}, \ref{thm:SBT log stability case r = (N-1)/2}, and \ref{thm:final stability theorem}.

    In Section \ref{Sec: Optimality}, we prove the optimality of Theorem \ref{thm:SBT stability} for domains that are at least of class $C^{1,1}$.
	
	Appendix \ref{sec:Optimality in the rough case} is devoted to proving the optimality of Theorem~\ref{thm:SBT stability} for $C^{1,\al}$ domains, when $\al<1$. 

    Appendices \ref{Proof of lemma of generalized 1.2} and \ref{Appendix: Proof of Lemma bound of Sobolev norms} are devoted to proving two auxiliary results that will be helpful in the proofs of Theorems \ref{thm:SBT stability}, \ref{thm:SBT log stability case r = (N-1)/2}, and \ref{thm:final stability theorem}.

	\section{Preliminaries}\label{sec:Preliminaries}
	We start by recalling the basic definitions of (possibly fractional) Sobolev spaces. Let $\Omega\subset \RR^{N}$ be an open set in $\RR^{N}$, $s\in (0,\infty)$ and $p\in[1,\infty]$. Write $s=m+\sigma$, where $m\in \NN$ and $\sigma \in [0,1)$ (if $m =0$, then we assume that $\sigma \in (0,1)$). \newline
    If $m = 0$ and $p<\infty$, the Sobolev space $W^{s,p}(\Omega)$ is defined as 
    $$
    W^{s,p}(\Omega):= \left\{u\in L^p(\Omega):\, |u|_{W^{s,p}(\Omega)}^p:=\int_{\Omega}\int_{\Omega}\frac{|u(x)-u(y)|^p}{|x-y|^{N+\sigma p}}\,dy\,dx <\infty \right\}, 
    $$
    normed with 
    $$
    \|u\|_{W^{s,p}(\Omega)} := \|u\|_{L^p(\Omega)} + |u|_{W^{s,p}(\Omega)}.
    $$
    If $m = 0$ and $p= \infty$, the Sobolev space $W^{s,\infty}(\Omega)$ is the space of $s-$Hölder continuous functions $C^{s}$, normed with
    $$
    \|u\|_{C^{s}(\Omega)} := \|u\|_{L^{\infty}(\Omega)} + \sup_{\substack{x,y\in \Omega\\ x\neq y}}\frac{|u(x)-u(y)|}{|x-y|^{\sigma}}
    $$
    If $m>0$ and $\sigma = 0$, then the Sobolev space $W^{s,p}(\Omega)$ is the standard integer Sobolev space $W^{m,p}(\Omega)$. \newline
    If $m>0$ and $\sigma \in (0,1)$, the Sobolev space $W^{s,p}(\Omega)$ is defined as 
    $$
    W^{s,p}(\Omega):= \left\{u\in W^{m,p}(\Omega):\, D^m u\in W^{\sigma, p}(\Omega) \right\},
    $$
    normed with 
    $$
    \|u\|_{W^{s,p}(\Omega)} := \|u\|_{W^{m,p}(\Omega)} + |D^m u|_{W^{\sigma,p}(\Omega)}.
    $$
    For an excellent reference on fractional Sobolev spaces see \cite{Hitchikers}.
    \bigskip
    
    We adopt the following definition of a $C^{k,\alpha}$ domain, where $k\ge 0$ is an integer, and $0\le\al\le 1$ (see e.g., \cite[Page 94]{GT}).
    \begin{definit}\label{def:regularity of a domain}
        Given an integer $k\ge 0$, a real number $0\le \al\le 1$ and a bounded domain $\Om \subset \mathbb{R}^N$ ($N\geq 2$) with boundary $\Ga$. We say that $\Ga$ is of class $C^{k,\al}$, if there exist positive constants $\tilde{C}$ and $\tilde{r}$ such that for any $x\in \Ga$, there exists a one-to-one map 
            \begin{equation*}
                \psi_{x}: B_{\tilde{r}}(x) \rightarrow D \subset \mathbb{R}^N,  
            \end{equation*}
            such that 
            \begin{align*}
                &\psi_{x}(B_{r}(x)\cap \Om) \subset \mathbb{R}_{+}^N , \quad \psi_{x}(B_{r}(x)\cap \Ga) \subset \partial\mathbb{R}_{+}^N,\\
                &\|\psi_x\|_{C^{k,\al}\left(B_{r}(x)\right)}\le \tilde{C} \quad \text{and}\quad \|\psi_{x}^{-1}\|_{C^{k,\al}(D)}\leq \tilde{C}.
            \end{align*}
            Here above, $B_{\tilde{r}}(x)$ denotes the $N-$dimensional euclidean ball centered at $x\in \mathbb{R}^N$ with radius $\tilde{r}$. 

            In particular, for every $x\in \Ga$, $B_{\tilde{r}}(x)\cap \Ga$ can be written as the graph of a $C^{k,\al}$ function of $N-1$ variables, whose $C^{k,\al}$ norm is bounded above by $\tilde{C}$. The converse is also true if $k\ge 1$.
    \end{definit}
     
    Throughout the paper, when we say that a constant depends on the $C^{k,\al}$ regularity of $\Ga$, it means that it can be estimated in terms of the constants $\tilde{C}$ and $\tilde{r}$ in Definition \ref{def:regularity of a domain}.

    \bigskip
    
    As mentioned in the Introduction, one of the steps in the proofs of Theorems \ref{thm:SBT stability}, \ref{thm:SBT log stability case r = (N-1)/2} and \ref{thm:final stability theorem} is the observation that if the deviation $\nr H-H_0\nr_{L^r(\Ga)}$ is small enough and $\Ga$ is of class $C^{k,\alpha}$ ($k\geq 1$ and $0<\alpha\leq 1$), then $\Om$ must be nearly spherical. This is done in Lemma \ref{lem:reduction to nearly spherical sets} below and the key to the proof of Lemma \ref{lem:reduction to nearly spherical sets} is the combination of the regularity properties of $\Ga$ with the following Lemma which provides a more general but rougher version of \eqref{eq:MP L2 + Pogcones} which deals with $L^r$ deviations of the mean curvature.
    
	\begin{lem}\label{lem: Generalization of 1.2}
    Let $r\in (1,\infty)$ and $\alpha\in (0,1]$, and let $\Omega\subset \mathbb{R}^N$ be a bounded domain with boundary $\Ga$ of class $C^{1,\alpha}\cap W^{2,r}$. Let $z\in \mathbb{R}^N$ be the center of mass of $\Om$, that is, 
    $$
    z = \frac{1}{|\Om|}\int_{\Om} x \,dx.
    $$
    Then
		\begin{equation}\label{eq:rough version of 1.2}
			\left\||x-z|-\frac{1}{H_0}\right\|_{L^2(\Ga)}+\left\|\frac{\nu}{H_0}-(x-z)\right\|_{L^2(\Gamma)}\leq C\left\|H-H_0\right\|_{L^r(\Gamma)}^{1/2},
		\end{equation}
        where $\nu$ denotes the exterior unit normal to $\Gamma$ at $x$, and $C$ is a positive constant only depending on $N$, $r$, the diameter of $\Om$ and the $C^{1,\alpha}$ regularity of $\Ga$.
	\end{lem}
     For a proof of Lemma \ref{lem: Generalization of 1.2} we refer to Appendix \ref{Proof of lemma of generalized 1.2}.
     
    Note that, for $r\geq 2$, in the right-hand side of \eqref{eq:rough version of 1.2}, we can replace $\left\|H-H_0\right\|_{L^r(\Gamma)}^{1/2}$ with $\left\|H-H_0\right\|_{L^r(\Gamma)}$, simply by applying Hölder's inequality to \eqref{eq:MP L2 + Pogcones}. 
    We stress that while \eqref{eq:MP L2 + Pogcones} is a sharp linear estimate for $r=2$, here we decided to present a rougher estimate (with exponent $1/2$) which holds for any $r \in (1,\infty)$; the (non-sharp) exponent here is not an issue for our aim, as this estimate will be used only to obtain a qualitative information which allows to restrict the analysis to nearly spherical sets, whereas the sharp rate of stability will be established later in the proofs of the main theorems.

 The next result we need is a Gagliardo-Nirenberg-type interpolation inequality that will play a crucial role in the proofs of Theorems \ref{thm:SBT stability}, \ref{thm:SBT log stability case r = (N-1)/2}, and \ref{thm:final stability theorem}.
\begin{thm}[Theorem 1, \cite{BrezisMironescu}]\label{thm:GagliardoNirenbergInterpolation}
		Let $D\subset \RR^d$ ($d\geq 1$) be a bounded Lipschitz domain. 
		Let $p>1$, $sp>d$, and $1\le q \le p$. Then there exists a constant $C$ only depending on $d$, $s$, $p$, $q$, $D$, such that for any $v\in W^{s,p}(D)$,
		\begin{equation}\label{eq:GagliardoNirenbergInterpolation}
			\nr v \nr_{L^\infty(D)} \le C \nr v \nr_{W^{s,p}(D)}^{1-\te} \nr v \nr_{L^q(D)}^{\te},
		\end{equation}
		where
		\begin{equation}\label{eq:teta in GN interpolation}
			\te:= \frac{s-\frac{d}{p}}{ s-\frac{d}{p}+\frac{d}{q}}.
		\end{equation}
	\end{thm}
     This, in fact, is a particular case of \cite[Theorem 1]{BrezisMironescu}, where much more general inequalities are studied, and we refer to \cite{BrezisMironescu} for a proof of Theorem \ref{thm:GagliardoNirenbergInterpolation}.
     
    In particular, we will make use of the following particular case of Theorem~\ref{thm:GagliardoNirenbergInterpolation} with $p=\infty$, $q=rd/(d-2r)$ and $d\ge 3$.
	
	\begin{cor}\label{cor:application of GN}
		Let $D\subset \RR^d$ ($d\ge 3$) be a bounded Lipschitz domain, and let $s>0$ and $r\in \left[1,\frac{d}{2}\right)$.
		If $v\in W^{s,\infty}(D)$, then
		\begin{equation}\label{eq:application GN to omega}
			\nr v \nr_{L^\infty(D)} \le C \nr v \nr_{W^{s,\infty}(D)}^{1-\te} \nr v \nr_{L^{\frac{rd}{d-2r} }(D)}^{\te},
			\text{ with }
			\te := \frac{s}{s+\frac{d-2r}{r}},
		\end{equation}
		where the constant $C$ depends only on $d$, $s$, $p$, $r$ and $D$.
	\end{cor} 
    
	We will also need the following technical lemma, which gives a bound on the Sobolev norm of a function defined in $\RR^{N-1}$ in terms of the Sobolev norm of the same function transported to the sphere $\SS^{N-1}$ via the stereographic projection. \newline
	We believe this is a well-known result in the literature, but since we did not find a reference containing it, we present it here and provide a detailed proof in Appendix~\ref{Appendix: Proof of Lemma bound of Sobolev norms}.
	
		\begin{lem}\label{lem:bound of flat W^2,2 Sobolev norm}
			Let $N\geq 2$ be an integer and $p\in [1,\infty)$. Then, there exists positive constants $R$ dependent only on $N$ and $C$ dependent only on $N$ and $p$, such that 
		\begin{equation}\label{eq:bound of flat W^2,2 Sobolev norm}
		\|v\circ \iota ^{-1}\|_{W^{2,p}\left(B_{R}^{N-1}(0)\right)}\leq C\|v\|_{W^{2,p}(\SS^{N-1})}, \quad \text{for all}\quad  v\in W^{2,p}(\SS^{N-1}),
        \end{equation}
        where $\iota: \SS^{N-1}\setminus \{P\}\rightarrow \mathbb{R}^{N-1}$ denotes the (standard) stereographic projection from $P\in \SS^{N-1}$, and $B_{R}^{N-1}(0)\subset \RR^{N-1}$ is the $(N-1)$-dimensional euclidean ball centered at the origin with radius $R$. 

        In particular, \eqref{eq:bound of flat W^2,2 Sobolev norm} holds with 
        \begin{equation}\label{eq:explicit radius}
            R= \frac{1}{8(N-1)^2},
        \end{equation}
        and $C = 3 \cdot 2^{\frac{N+5}{p}+\frac{5}{2}}$.
	\end{lem}
	
	\section{Proof of Theorems \ref{thm:SBT stability}, \ref{thm:SBT log stability case r = (N-1)/2} and \ref{thm:final stability theorem}}\label{sec:Proof of Main Theorems}
	Before we proceed with the proofs of Theorems \ref{thm:SBT stability}, \ref{thm:SBT log stability case r = (N-1)/2} and \ref{thm:final stability theorem}, we show that we can restrict our analysis to nearly spherical sets.
	\begin{lem}\label{lem:reduction to nearly spherical sets}
		Let $N\geq 2$ be a integer and $r\in (1,\infty)$. Let $\Omega\subset \mathbb{R}^N$ be a bounded domain with boundary of class $C^{1,\alpha}\cap W^{2,r}$ ($0<\alpha\leq 1$). Set the origin at the center of mass of $\Om$, that is,
        \begin{equation}\label{eq:star001}
            \frac{1}{|\Om|} \int_\Om x \, dx = 0,
        \end{equation}
        which is always true up to a translation.

		Then, there exists a positive constant $\eps_0$ dependent only on $N$, $r$, the diameter of $\Om$ and the $C^{1,\alpha}$-regularity of $\Gamma$ such that if $\|H-H_0\|_{L^r(\Gamma)}< \eps_0$, then there exists $\omega:\SS^{N-1}\rightarrow \RR$ such that
		\begin{equation}\label{eq:conclusion of near spherical}
		\Gamma= \left\{\left(\frac{1}{H_0}+\omega(x)\right)x:\,x\in\SS^{N-1}\right\}.
		\end{equation}
		Furthermore, for any $\eta>0$ there exists $\eps_1>0$ dependent only on $N$, $r$, $\eta$, the diameter of $\Om$ and the $C^{1,\alpha}$-regularity of $\Gamma$ such that, if $\|H-H_0\|_{L^r(\Ga)}<\eps_1$ then 
		\begin{equation}\label{eq:estimate on C1 norm of omega}
			\|\omega\|_{C^{1}(\SS^{N-1})}\leq \eta.\\
		\end{equation}
	\end{lem}
	\begin{proof}
			By Lemma \ref{lem: Generalization of 1.2} and \eqref{eq:star001} we know that there exists a positive constant $C$, dependent on $N$, $r$, the diameter of $\Om$ and the $C^{1,\alpha}$ regularity of $\Gamma$ such that
		\begin{equation}\label{eq:generalized (1.1)}
			\left\||x|-\frac{1}{H_0}\right\|_{L^2(\Ga)}+\left\|\frac{\nu}{H_0}-x\right\|_{L^2(\Gamma)}\leq C\left\|H-H_0\right\|_{L^r(\Gamma)}^{1/2},
		\end{equation}
		where $\nu$ denotes the exterior unit normal to $\Gamma$ at $x$.
				
		Discarding the second term in~\eqref{eq:generalized (1.1)} we have
 		\begin{equation}\label{eq:star002}
 					\int_{\Ga}\left||y|-\frac{1}{H_0}\right|^2\,d\mathcal{H}_{y}^{N-1}\leq C^2 \|H-H_0\|_{L^r(\Ga)}.
 		\end{equation}
 		Let $\delta>0$, and define the measurable sets $S_{\delta}^+,S_{\delta}^{-}\subset \SS^{N-1}$ by setting
 		$$
 		S_{\delta}^+:=\left\{x\in\Ga:\left||y|-\frac{1}{H_0}\right|^2> \delta\right\},\quad \text{and}\quad S_{\delta}^- = \SS^{N-1}\setminus S_{\delta}^+.
 		$$
 		Then, we have that 
 		\begin{equation}\label{eq:estimate on the measure of the worst set}
 		|S_{\delta}^+|\leq \frac{1}{\delta}\int_{\Ga}\left||y|-\frac{1}{H_0}\right|^2\,d\mathcal{H}_{y}^{N-1}\leq \frac{C^2}{\delta} \|H-H_0\|_{L^r(\Ga)},
 		\end{equation}
        where in the last inequality we used \eqref{eq:star002}.
 		
 		Let $y\in S_{\delta}^+$. Since $\Ga$ is of class $C^{1,\alpha}$, making use of Definition \ref{def:regularity of a domain}, we see that there exist $r_1$ dependent only on the $C^{1,\alpha}$ regularity of $\Ga$, such that  
 		$$
 		\left(|w|-\frac{1}{H_0}\right)^2 >\frac{\delta}{2},\quad \text{for all } w\in \left\{z\in \Ga:\, |z-y|<r_1\right\},   
 		$$ 
 		and
 		$$
 		\left| \left\{z\in \Ga:\, |z-y|< r_1\right\}\right|\geq
        C_1 \, r_{1}^{N-1},
 		$$
        where $C_1 = \left|B_{1}^{N-1}\right|$ is the volume of the unit $(N-1)$-dimensional Euclidean ball in $\RR^{N-1}$. The last inequality follows directly from Definition \ref{def:regularity of a domain}, which allows us to write $\left\{z\in \Ga:\, |z-y|< r_1\right\}$ as the graph of a $C^{1,\al}$ function defined on a $(N-1)$-dimensional euclidean ball in $\mathbb{R}^{N-1}$ with radius $r_1$, and
        the formula for the area of a surface (see, e.g., \cite{EG}). \newline
 		But then we conclude that
 		$$
 		|S_{\delta/2}^+| \geq C_1 r_{1}^{N-1}.
 		$$
 		Recalling \eqref{eq:estimate on the measure of the worst set}, we obtain 
 		$$
 		C_1 r_{1}^{N-1}\leq |S_{\delta/2}^+|\leq \frac{2C^2}{\delta}\|H-H_0\|_{L^r(\Ga)}.
 		$$
 		So, setting $\tilde{\eps}_1(\delta) :=  \frac{\delta C_1}{2C^2}r_{1}^{N-1}$, we see that if 
 		$$
 		\|H-H_0\|_{L^r(\Ga)}< \tilde{\eps}_1(\delta)
 		$$
 		then we conclude that $S_{\delta}^+ =\emptyset$.
 		
 		Similarly, we conclude that for any $\epsilon>0$ there exists $\tilde{\eps}_2(\epsilon)>0$ dependent on the $C^{1,\alpha}$ regularity of $\Ga$, $N$, $r$ and $\epsilon$ such that, if $\|H-H_0\|_{L^r(\Ga)}< \tilde{\eps_2}(\epsilon)$, then
 		$$
 		\left\{y\in \Ga: \left|\frac{\nu(y)}{H_0}-y\right|^2> \epsilon\right\} = \emptyset.
 		$$
 		
 		In particular, if we take $\epsilon < \frac{1}{H_0}$ then, if $\|H-H_0\|_{L^r(\Ga)}< \tilde{\eps}_2(\epsilon)$, we have that $\nu(y)\cdot y >0$ for all $y\in\Ga$. This implies that $\Omega$ is star-shaped about the origin, and therefore (up to a translation), we can write 
 		\begin{equation}\label{eq:near spherical}
 			\Ga =\left\{\left(\frac{1}{H_0}+\omega(x)\right)x:\,x\in\SS^{N-1}\right\},
 		\end{equation}
 		for some function $\omega:\SS^{N-1}\rightarrow \RR$ satisfying
        \begin{equation}\label{eq:star0001}
        \omega (x)\geq -\frac{1}{H_0}.
        \end{equation}
        So, taking $\eps_0:=\tilde{\eps}_2\left(\frac{1}{2H_0}\right)$, \eqref{eq:conclusion of near spherical} follows. We note that the dependence of $\eps_0$ on $H_0$ can be removed using Lemma \ref{lem:bound on surface measure Ga} below.\footnote{In fact, by Lemma \ref{lem:bound on surface measure Ga} and the definition of $H_0$, we have that $\ka \le \frac{1}{H_0}$, where $\ka$ is the constant appearing in the statement of Lemma \ref{lem:bound on surface measure Ga}, which only depends on the $C^{1,\al}$ regularity of $\Ga$.}
        In this way, $\eps_0$ depends on $N$, $r$, the diameter of $\Om$ and the $C^{1,\al}$ regularity of $\Ga$.
 		
 		Notice that since $\Ga$ is of class $C^{1,\alpha}$, then $\omega$ is in $C^{1,\alpha}(\SS^{N-1})$ and $\|\omega\|_{C^{1,\alpha}(\SS^{N-1})}$ depends only on the $C^{1,\alpha}$ regularity of $\Ga$.
 		
 		Now we prove the second part of the statement. To that end, let $\sigma>0$ be such that
        \begin{equation}\label{eq:star0002}
            \sigma<<\frac{1}{H_0},
        \end{equation}
         and let $\eps_1 = \min\{\tilde{\eps}_1(\sigma), \tilde{\eps}_2(\sigma)\}>0$. Then, by the arguments above, if $\|H-H_0\|_{L^r(\Ga)}<\eps_1$, then
 		\begin{equation}\label{eq:emptysets}
 			\left\{y\in \Ga:\, \left||y|-\frac{1}{H_0}\right|^2> \sigma\right\}=\emptyset, \quad \text{and}\quad \left\{y\in \Ga: \left|\frac{\nu(y)}{H_0}-y\right|^2> \sigma\right\} = \emptyset.
 		\end{equation} 
 		Note that, due to the arguments presented above, we know that $\eps_1$ will depend only on the $C^{1,\alpha}$ regularity of $\Ga$, $N$, $r$, and $\sigma$. Furthermore, \eqref{eq:near spherical} holds for some $C^{1,\alpha}$ function $\omega:\SS^{N-1}\rightarrow \RR$.
 		
 		Using \eqref{eq:near spherical}, given $y = \left(\frac{1}{H_0}+\omega(x)\right)x$, the outward unit normal at $y$, $\nu(y)$, can be written as (see \cite{Urbas})
 		\begin{equation}\label{eq:outward unit normal in near spherical}
 			\nu(y) = \nu\left(\left(\frac{1}{H_0}+\omega(x)\right)x\right) = \frac{\left(\frac{1}{H_0}+\omega(x)\right)x-\nabla_{\SS^{N-1}}\omega(x)}{\sqrt{\left(\frac{1}{H_0}+\omega(x)\right)^2+|\nabla_{\SS^{N-1}}\omega(x)|^2}}.
 		\end{equation} 
 		
 		Combining \eqref{eq:near spherical} and \eqref{eq:emptysets} we see that 
 		\begin{equation}
 			\left|\left|\frac{1}{H_0}+\omega(x)\right|-\frac{1}{H_0}\right|\leq \sigma^{1/2}, \quad \text{for all} \, x\in \SS^{N-1},
 		\end{equation}
 		and this implies that 
 		\begin{equation*}
 			|\omega(x)|\leq \sigma^{1/2},\quad \text{for all} \, x\in \SS^{N-1}, \quad \text{or}\quad |\omega(x)+\frac{2}{H_0}|\leq \sigma^{1/2},\quad \text{for all} \, x\in \SS^{N-1}.
 		\end{equation*}
        Combining this with \eqref{eq:star0001} and \eqref{eq:star0002} yields
        \begin{equation}\label{eq:uniform bound for omega}
            |\omega(x)|\leq \sigma^{1/2}, \quad \text{for all} \, x\in \SS^{N-1}.
        \end{equation}
 		
 		Combining \eqref{eq:near spherical}, \eqref{eq:emptysets},\eqref{eq:outward unit normal in near spherical} and \eqref{eq:uniform bound for omega} and the fact that $x\cdot \nabla_{\SS^{N-1}}\omega(x) = 0$ for all $x\in \SS^{N-1}$, we obtain 
 		\begin{equation}
 			\begin{split}
 			 \frac{1}{H_{0}^2}+\left(\frac{1}{H_0}+\omega(x)\right)^2 -\frac{2}{H_0} \frac{\left(\frac{1}{H_0}+\omega(x)\right)^2}{\sqrt{\left(\frac{1}{H_0}+\omega(x)\right)^2+|\nabla_{\SS^{N-1}}\omega(x)|^2}}\leq \sigma.
 			\end{split}
 		\end{equation}
 		This combined with \eqref{eq:uniform bound for omega} leads to 
 		\begin{equation}\label{eq:uniform gradient bound for omega}
 			\left|\nabla_{\SS^{N-1}}\omega(x)\right|\leq C_2 \sigma^{1/4},
 		\end{equation} 		
 		where $C_2$ is a positive constant dependent only on $H_0$.
        As usual, we can remove the dependence on $H_0$ recalling Lemma \ref{lem:bound on surface measure Ga}.
        
 		Combining \eqref{eq:uniform bound for omega} and \eqref{eq:uniform gradient bound for omega}, \eqref{eq:estimate on C1 norm of omega} easily follows by adjusting $\sigma$.
 		
	\end{proof}
	
	We are now ready for the
	\begin{proof}[Proof of Theorem \ref{thm:SBT stability}]
    We set the origin at the center of mass of $\Om$, that is,
    \begin{equation*}
        \frac{1}{|\Om|} \int_\Om x \, dx = 0,
    \end{equation*}
    which is always true up to a translation.
    	
    Start by assuming that $\|H-H_0\|_{L^r(\Gamma)}<\eps_0$, where $\eps_0$ is given by the previous lemma. 
	Then, there exists $\omega:\SS^{N-1}\rightarrow \RR$ such that 
		$$
		\Gamma = \left\{\left(\frac{1}{H_0}+\omega(x)\right)x:x\in\SS^{N-1}\right\},
		$$
		$\omega\in C^{k,\alpha}(\SS^{N-1})\cap W^{2,r}(\SS^{N-1})$ and $\omega(x)\geq -\frac{1}{H_0}$, for all $x\in \SS^{N-1}$. 
		Then, we have that 
		\begin{equation}\label{eq:bound on diference of radii}
			\rho_e(0)-\rho_i(0)\leq 2 \|\omega\|_{L^{\infty}(\SS^{N-1})}.
		\end{equation}
		Let $x_0\in \SS^{N-1}$ be such that $|\omega(x_0)|=\|\omega\|_{L^{\infty}(\SS^{N-1})}$.
		Assume that $x_0 = (0,...,0,-1)$, which, up to rotations, is true. In this way, we have that $\iota(x_0)=0$, where $\iota$ denotes the standard stereographic projection from $P=(0,...,0,1)$ (see~\eqref{eq:stereographic projection from the north pole} for an explicit formula). Let $R$ be the constant defined in \eqref{eq:explicit radius} (which depends only on $N$). Applying Corollary ~\ref{cor:application of GN} (in $\RR^{N-1}$) with $D = B_R^{N-1}(0)$, $s=k+\alpha$ and $v=\omega\circ \iota^{-1}$ we obtain
	\begin{equation}\label{eq:GN applied to omega}
		\nr \omega\nr_{L^{\infty}(\SS^{N-1})}= \nr \omega\circ \iota^{-1} \nr_{L^\infty\left(B_{R}^{N-1}(0)\right)} \le C_1 \nr \omega\circ \iota^{-1} \nr_{W^{k+\alpha,\infty}\left(B_{R}^{N-1}(0)\right)}^{1-\te} \nr \omega\circ \iota^{-1} \nr_{L^{\frac{r(N-1)}{N-1-2r}}\left(B_{R}^{N-1}(0)\right)}^{\te},
	\end{equation}
	where $C_1$ only depends\footnote{We point out that the dependence of the constant $C_1$ on $D=B_{R}^{N-1}(0)$ can be dropped because $B_{R}^{N-1}(0)$ does not depend on $\Om$, but it is a ball in $\mathbb{R}^{N-1}$ with radius $R$ which only depends on $N$.} on $N$, $k$, $\alpha$ and $r$ and where
	\begin{equation}\label{eq:stability exponent}
		\te := \frac{k+\alpha}{k+\alpha +\frac{N-1-2r}{r}}.
	\end{equation}
	
	By the Sobolev embedding Theorem (see~\cite[Chapter 4]{AdamsFournier}) we have that 
	$$
	 \nr \omega\circ \iota^{-1} \nr_{L^{\frac{r(N-1)}{N-1-2r}}\left(B_{R}^{N-1}(0)\right)}\leq C_2 \left\|\omega\circ \iota^{-1}\right\|_{W^{2,r}\left(B_{R}^{N-1}(0)\right)},
	$$ 
	where $C_2$ is a positive constant dependent only on $N$ and $r$. This leads to
    \begin{equation}\label{eq:bound with W 22 norm}
        \|\omega\|_{L^{\infty}(\SS^{N-1})}\leq  C_3 \nr \omega\circ \iota^{-1} \nr_{W^{k+\al,\infty}\left(B_{R}^{N-1}(0)\right)}^{1-\te} \nr \omega\circ \iota^{-1} \nr_{W^{2,r} \left(B_{R}^{N-1}(0)\right)}^{\te},
    \end{equation}
    where $C_3 = C_{1}C_{2}^{\theta}$. 

    By \cite[Theorem $1.6$]{FZ}\footnote{Note that by \eqref{eq:star001} and \eqref{eq:estimate on C1 norm of omega} in Lemma \ref{lem:reduction to nearly spherical sets} we can apply~\cite[Theorem $1.6$]{FZ}.} we know that there exists a positive constant $C_4$ dependent only $N$, $r$ and the $C^{1,\alpha}$ regularity of $\Gamma$ such that 
        \begin{equation}\label{eq:regularity bound for mean curvature equation}
            \|\omega\|_{W^{2,r}(\SS^{N-1})}\leq C_4\|H-H_0\|_{L^{r}(\Gamma)}.
        \end{equation}

    Combining \eqref{eq:bound on diference of radii}, \eqref{eq:bound with W 22 norm}\footnote{We point out that $\nr \omega\circ \iota^{-1} \nr_{W^{k+\alpha,\infty}\left(B_{R}^{N-1}(0)\right)}$ depends solely on the regularity of $\Ga$. In the case $\al =1$ we point out that since $\omega\circ \iota^{-1}\in C^{k,1}\left(B_{R}^{N-1}(0)\right)$, then we also have $\omega\circ \iota^{-1} \in W^{k+1,\infty}\left(B_{R}^{N-1}(0)\right)$ and the $W^{k+1,\infty}$ norm of $\omega\circ\iota^{-1}$ will depend only on the $C^{k,1}$ norm of $\omega\circ \iota^{-1}$, see \cite{Evans} for additional details.}, Lemma \ref{lem:bound of flat W^2,2 Sobolev norm} (with $v= \om$ and $p=r$) and \eqref{eq:regularity bound for mean curvature equation}, we obtain~\eqref{eq:stability SBTGN}-~\eqref{eq:stability exponent I}, provided that $\nr H-H_0\nr_{L^r(\Gamma)}<~\eps_0$.	
      
    If $\nr H-H_0\nr_{L^r(\Gamma)}\geq\eps_0$, the result trivially follows noting that, for any $z \in \Om$,
    \begin{equation*}
        \rho_e(z)-\rho_i(z)\leq d_{\Omega}=\frac{d_{\Omega}}{\eps_{0}^{\theta}}\eps_{0}^{\theta}\leq \frac{d_{\Omega}}{\eps_{0}^{\theta}}\nr H-H_0\nr_{L^r(\Gamma)}^{\theta},
    \end{equation*}
    where $d_{\Omega}$ is the diameter of $\Omega$.
	\end{proof}

    \bigskip
    
	\begin{proof}[Proof of Theorem~\ref{thm:SBT log stability case r = (N-1)/2}]
    We set the origin at the center of mass of $\Om$, that is,
    \begin{equation*}
        \frac{1}{|\Om|} \int_\Om x \, dx = 0,
    \end{equation*}
    which is always true up to a translation.
    
		Now we assume that $\|H-H_0\|_{L^{\frac{N-1}{2}}(\Gamma)}<\eps_0$, where $\eps_0$ is given by Lemma~\ref{lem:reduction to nearly spherical sets}. 
		Then there exists $\omega:\SS^{N-1}\rightarrow \RR$ such that 
		$$
		\Gamma = \left\{\left(\frac{1}{H_0}+\omega(x)\right)x:x\in\SS^{N-1}\right\},
		$$
		$\omega\in C^{k,\alpha}(\SS^{N-1})\cap W^{2,r}(\SS^{N-1})$ and $\omega(x)\geq -\frac{1}{H_0}$, for all $x\in \SS^{N-1}$. 
        As a consequence, we have that
        \begin{equation}\label{eq:difference of radii}
            \rho_e(0)-\rho_i(0)\leq 2 \|\omega\|_{L^{\infty}(\SS^{N-1})}.
        \end{equation}
        Let $x_{0}\in \SS^{N-1}$ be such that
        \begin{equation*}
            |\omega(x_0)| =  \|\omega\|_{L^{\infty}(\SS^{N-1})}.
        \end{equation*}
        Similarly to what we did in the proof of Theorem \ref{thm:SBT stability}, assume that $x_0= (0,...,0,-1)$, which, up to rotations, is true. In this way, we have $\iota(x_0)=0$, where $\iota$ denotes the standard stereographic projection from $P=(0,...,0,1)$ (see \eqref{eq:stereographic projection from the north pole} for an explicit formula). Let $R$ be the constant defined in \eqref{eq:explicit radius} (which only depends on $N$). 
        Then, applying \cite[Lemma 2.1]{MP6} and \cite[Lemma 2.5 (ii)]{MP6} (in $\mathbb{R}^{N-1}$) with $f=\omega\circ \iota^{-1}$, $p=N-1$, $q=\infty$ and $\mathcal{C}$ a cone in $\mathbb{R}^{N-1}$ with vertex $0\in \mathbb{R}^{N-1}$, opening width $\pi/4$ and height $R/2$ we obtain the following
        \begin{equation*}
        \left|\omega\circ \iota^{-1}(0)-(\omega\circ \iota^{-1})_{\mathcal{C}}\right|\le \frac{R}{N |\cC|^{1/(N-1)}} \, \nr \na \left(\omega\circ \iota^{-1}\right) \nr_{L^{N-1}\left(\mathcal{C}\right)} \log \left( e \,  \frac{ \left|\mathcal{C}\right|^{\frac{1}{N-1}} \nr \na \left(\omega\circ \iota^{-1}\right) \nr_{L^\infty \left(\mathcal{C}\right)}}{\nr \na \left(\omega\circ \iota^{-1}\right) \nr_{L^{N-1}\left(\mathcal{C}\right)}} \right),
        \end{equation*}
        where $(\omega\circ\iota^{-1})_{\mathcal{C}} = \frac{1}{|\mathcal{C}|}\int_{\mathcal{C}}\omega\circ \iota^{-1}(y)\,dy$.

        By the triangle inequality, and using that $\om\circ \iota^{-1}(0) = \om(x_0) = \nr \om \nr_{L^\infty(\SS^{N-1})}$, we thus have that
        \begin{equation*}
            \nr \om \nr_{L^\infty(\SS^{N-1})} \le |(\om\circ \iota^{-1})_\cC| +  \frac{R}{N|\cC|^{1/(N-1)}} \, \nr \na \left(\omega\circ \iota^{-1}\right) \nr_{L^{N-1}\left(\mathcal{C}\right)} \log \left( e \,  \frac{ \left|\mathcal{C}\right|^{\frac{1}{N-1}} \nr \na \left(\omega\circ \iota^{-1}\right) \nr_{L^\infty \left(\mathcal{C}\right)}}{\nr \na \left(\omega\circ \iota^{-1}\right) \nr_{L^{N-1}\left(\mathcal{C}\right)}} \right).
        \end{equation*}
        Using that $\cC \subset B_R^{N-1} (0)$, the monotonicity, for any $A > 0$ and $t > 0$, of the function $t \to t \max\{\log(A/t), 1\}$, and the fact that $(\om\circ \iota^{-1})_\cC \le \frac{1}{|\cC|^{1/(N-1)}} \left\nr  \om\circ \iota^{-1} \right\nr_{L^{N-1}\left(B_{R}^{N-1}(0)\right)}$, we readily obtain
        \begin{equation}\label{eq:1 for proof of r = (N-1)/2}
            \nr \om \nr_{L^\infty(\SS^{N-1})}\leq C_1 \nr\omega\circ \iota^{-1}\nr_{W^{1,N-1}\left(B^{N-1}_R(0)\right)}\max\left\{\log\left(e\,\frac{ \left|\mathcal{C}\right|^{\frac{1}{N-1}} \nr \na \left(\omega\circ \iota^{-1}\right) \nr_{L^\infty \left(B^{N-1}_R(0)\right)}}{\nr\omega\circ \iota^{-1}\nr_{W^{1,N-1}\left(B^{N-1}_R(0)\right)}} \right),1\right\},
        \end{equation}
        where $C_1$ is a positive constant dependent only on $N$ (recall that $|\mathcal{C}|$ depends only on $N$, by construction).
		
        Set $\ell:= \frac{\na \left(\omega\circ \iota^{-1}\right)(x_{M})}{|\na \left(\omega\circ \iota^{-1}\right)(x_M)|}$ where $x_M\in\ol{B_{R}^{N-1}(0)}$ is a point where $|\na \left(\omega\circ \iota^{-1}\right)|$ attains its maximum in $\ol{B_{R}^{N-1}(0)}$; applying Theorem \ref{thm:GagliardoNirenbergInterpolation} (in $\RR^{N-1}$) with $D=B_{R}^{N-1}(0)$, $v= \langle \na \left(\omega \circ \iota^{-1}\right), \ell \rangle$, $p=\infty$ and $q=N-1$ gives that, for any $s>0$,
		\begin{equation*}
			\left\nr \na \left(\omega \circ \iota^{-1}\right) \right\nr_{L^\infty\left(B_{R}^{N-1}(0)\right)} \le C_2 \left\nr \na \left(\omega\circ \iota^{-1}\right) \right\nr_{W^{s,\infty}\left(B_{R}^{N-1}(0)\right)}^{\frac{1}{s+1}} \left\nr \na \left(\omega\circ \iota^{-1}\right) \right\nr_{L^{N-1} \left(B_{R}^{N-1}(0)\right) }^{\frac{s}{s+1}}, 
		\end{equation*}
		and hence
		\begin{equation}\label{eq:2 for proof of case r = (N-1)/2}
			\left\nr \na\left(\omega\circ \iota^{-1}\right) \right\nr_{L^\infty\left(B_{R}^{N-1}(0)\right)} \le C_2 \left\nr  \omega\circ \iota^{-1}\right\nr_{W^{s+1,\infty}\left(B_{R}^{N-1}(0)\right)}^{\frac{1}{s+1}} \left\nr \omega\circ \iota^{-1}\right\nr_{W^{1,N-1} \left(B_{R}^{N-1}(0)\right) }^{\frac{s}{s+1}},
		\end{equation}
		where the constant $C_2$ depends\footnote{As already noticed, we can drop the dependence on $D = B_{R}^{N-1}(0)$.} only on $s$ and $N$. 
		
		Combining \eqref{eq:difference of radii}, \eqref{eq:1 for proof of r = (N-1)/2} and \eqref{eq:2 for proof of case r = (N-1)/2} easily leads to
		\begin{equation}\label{eq:3 for proof of the case r = (N-1)/2}
			\rho_e(0) - \rho_i(0)  \le C_3 \,  \nr \omega\circ \iota^{-1}\nr_{W^{1,N-1}\left(B_{R}^{N-1}(0)\right)}
			\max \left\{ \frac{1}{s+1} \,  \log \left( \frac{ \left\nr \left(\omega\circ \iota^{-1}\right)\right\nr_{W^{s+1,\infty}\left(B_{R}^{N-1}(0)\right)} }{\nr \omega\circ \iota^{-1} \nr_{W^{1,N-1}\left(B_{R}^{N-1}(0)\right)}}  \right) , 1 \right\},
		\end{equation}
		where the constant $C_3$ depends only on $s$ and $N$.
		
		Applying the Sobolev embedding Theorem (see~\cite[Chapter 4]{AdamsFournier}), Lemma~\ref{lem:bound of flat W^2,2 Sobolev norm} and \cite[Theorem $1.6$]{FZ}\footnote{As before, we note that by \eqref{eq:star001} and \eqref{eq:estimate on C1 norm of omega} in Lemma \ref{lem:reduction to nearly spherical sets} we can apply~\cite[Theorem $1.6$]{FZ}.}, we obtain 
		\begin{equation}\label{eq: Sobolev + regularity theory}
			\nr\omega\circ \iota^{-1}\nr_{W^{1,N-1}\left(B_{R}^{N-1}(0)\right)}\leq C_4 \nr \omega\circ \iota^{-1}\nr_{W^{2,\frac{N-1}{2}}(\SS^{N-1})}\leq C_5\nr H-H_0\nr_{L^{\frac{N-1}{2}}(\Gamma)},
		\end{equation}
		where $C_4$ is a positive constant dependent only on $N$, and $C_5$ is a positive constant dependent on $N$ and the $C^{1,\alpha}$ regularity of $\Gamma$.
		
		So, combining~\eqref{eq:3 for proof of the case r = (N-1)/2}, \eqref{eq: Sobolev + regularity theory} and the monotonicity, for any $A>0$ and $t>0$, of the function $t~\mapsto~t \max \left\{\log\left(A/ t \right) , 1 \right\}$, we obtain 
		\begin{equation}\label{eq:prefinal case r=(N-1)/2 proof}
			\rho_e(0) - \rho_i(0)  \le C_6 \,  \nr H_0 - H \nr_{L^{\frac{N-1}{2}}(\Ga)}
			\max \left\{ \frac{1}{ s+1}\, \log \left( \frac{ \left\nr \omega\circ \iota^{-1} \right\nr_{W^{s+1,\infty}\left(B_{R}^{N-1}(0)\right)} }{\nr H_0 - H \nr_{L^{\frac{N-1}{2}}(\Ga)}}  \right) , 1 \right\},	
		\end{equation}
		where $s>0$, and the constant $C_6$ only depends on $s$, $N$ and the $C^{1,\al}$-regularity of $\Ga$.	
        
        Finally, we apply~\eqref{eq:prefinal case r=(N-1)/2 proof} with $s+1= k+\alpha$ to obtain~\eqref{eq:case r=(N-1)/2 log stability SBT via Holder estimates}, provided that $\nr H-H_0\nr_{L^{\frac{N-1}{2}}(\Gamma)}<\eps_0$.\newline
		This finishes the proof under the assumption that $\|H-H_0\|_{L^{\frac{N-1}{2}}(\Ga)}< \eps_0$.
     \bigskip
     
     If $\nr H-H_0\nr_{L^{\frac{N-1}{2}}(\Gamma)}\geq\eps_0$, then it is clear that~\eqref{eq:case r=(N-1)/2 log stability SBT via Holder estimates} still holds since, for any $z\in \Om$,
        $$
        \rho_e(z) - \rho_i(z)\le d_{\Omega} \le \frac{d_{\Omega}}{\eps_0} \|H-H_0\|_{L^{\frac{N-1}{2}}(\Ga)},
        $$
        where $d_{\Omega}$ denotes the diameter of $\Omega$.
	\end{proof}
	
	\bigskip
    We finalise this section with the proof of Theorem \ref{thm:final stability theorem}.
    
	\begin{proof}[Proof of Theorem \ref{thm:final stability theorem}]
	    We start by noting that the only case left to prove is the one where $r>\frac{N-1}{2}$.
        Indeed, if $N\leq 3$ then we have $r>1>\frac{N-1}{2}$. If $N\geq 4$, the case $r\leq\frac{N-1}{2}$ has already been covered by Theorems \ref{thm:SBT stability} and \ref{thm:SBT log stability case r = (N-1)/2}, so that we are left with the case $r>\frac{N-1}{2}$. 

        Now let $r>\frac{N-1}{2}$. Similarly to what we did in the previous proofs, we set the origin at the center of mass of $\Om$, that is,
    \begin{equation*}
        \frac{1}{|\Om|} \int_\Om x \, dx = 0,
    \end{equation*}
    which is always true up to a translation. We also assume that $\|H-H_0\|_{L^r(\Gamma)}<\eps_0$, where $\eps_0$ is given by Lemma \ref{lem:reduction to nearly spherical sets}. 
	Then, there exists $\omega:\SS^{N-1}\rightarrow \RR$ such that 
		$$
		\Gamma = \left\{\left(\frac{1}{H_0}+\omega(x)\right)x:x\in\SS^{N-1}\right\},
		$$
		 $\omega\in C^{k,\alpha}(\SS^{N-1})\cap W^{2,r}(\SS^{N-1})$ and $\omega(x)\geq -\frac{1}{H_0}$.
		Then, we have that 
		\begin{equation}\label{eq:bound on diference of radii 2}
			\rho_e(0)-\rho_i(0)\leq 2 \|\omega\|_{L^{\infty}(\SS^{N-1})}.
		\end{equation}
        Let $x_0\in \SS^{N-1}$ be such that $|\omega(x_0)|=\|\omega\|_{L^{\infty}(\SS^{N-1})}$.
		Assume that $x_0 = (0,...,0,-1)$, which, up to rotations, is true. In this way, we have that $\iota(x_0)=0$, where $\iota$ denotes the standard stereographic projection from $P=(0,...,0,1)$ (see~\eqref{eq:stereographic projection from the north pole} for an explicit formula). Let $R$ be the constant defined in \eqref{eq:explicit radius} (which depends only on $N$).

        Since $r>\frac{N-1}{2}$, we can apply the Sobolev embedding Theorem (see \cite[Chapter 4]{AdamsFournier}) to conclude that there exists a positive constant $C_1$ dependent only on $N$ and $r$, such that 
        \begin{equation*}
            \|\omega\|_{L^{\infty}(\SS^{N-1})} = \left\|\omega\circ \iota^{-1}\right\|_{L^{\infty}\left(B_{R}^{N-1}(0)\right)}\leq C_1 \|\omega\circ \iota^{-1}\|_{W^{2,r}\left(B_{R}^{N-1}(0)\right)}.
        \end{equation*}
        This, combined with \eqref{eq:bound on diference of radii 2}, Lemma \ref{lem:bound of flat W^2,2 Sobolev norm} and \cite[Theorem $1.6$]{FZ}, leads to 
        \begin{equation*}
            \rho_e(0)-\rho_i(0) \leq C_2 \|H-H_0\|_{L^r\left(\Ga\right)},
        \end{equation*}
        where $C_2$ is a positive constant dependent only $N$, $r$ and the $C^{k,\alpha}$ regularity of $\Ga$. 
        This concludes the proof of \eqref{eq:final stability} under the assumption that $\|H-H_0\|_{L^r\left(\Ga\right)}<\eps_{0}$. 

        If $\|H-H_0\|_{L^r\left(\Ga\right)}\geq \eps_{0}$, then \eqref{eq:final stability} still holds since, for any $z\in \Om$,
        $$
        \rho_e(z) - \rho_i(z)\le \frac{d_{\Omega}}{\eps_0} \|H-H_0\|_{L^{r}(\Ga)},
        $$
        where $d_{\Omega}$ denotes the diameter of $\Omega$.
        \end{proof}
        \bigskip
        
	\section{Optimality of Theorem~\ref{thm:SBT stability} with $k+ \alpha \geq 2$}\label{Sec: Optimality}
	In this section, we show the optimality of Theorem~\ref{thm:SBT stability} for $C^{k,\alpha}$ domains, when $k+\alpha \geq 2$ ($0< \al\le 1$).
	Since the estimates to check the optimality of Theorems~\ref{thm:SBT stability} are slightly different and more technical for $C^{1,\alpha}$ ($\alpha\in \left(0,1\right)$) domains, we leave the $C^{1,\alpha}$ case for the Appendix (see Section~\ref{sec:Optimality in the rough case}).
	
	Throughout this section, $B_{r}^{N-1}$ denotes the $(N-1)$-dimensional ball of radius $r$ with center at the origin and orthogonal to $e_{N}=(0,...,0,1)\in\mathbb{R}^N$. Let $N\geq 4$, $k\geq 1$ be nonnegative integers and $\alpha \in (0,1]$ be, such that $k+\alpha \geq 2$ (if $k=1$, the only option is $\al =1$). 
	
	Define $\Gamma_{t}\subset \mathbb{R}^{N}$ (for $t>0$) by setting 
	\begin{gather*}
		\Gamma_{t}\cap \left(B_1^{N-1}\times (0,+\infty)\right) = \{(x,\varphi_{t}(x)):\,x\in B_1^{N-1}\},\\
		\Gamma_{t}\setminus \left(B_1^{N-1}\times (0,+\infty)\right) =\mathbb{S}^{N-1}\setminus  \left(B_1^{N-1}\times (0,+\infty)\right),
	\end{gather*}
	where $\varphi_{t}:B_{1}^{N-1}\rightarrow (0,\infty)$ is a function of class $C^{k,\alpha}$.\\
	Now we define $\varphi_{t}$. First, let $\varphi_{0}:B_1^{N-1}\rightarrow \mathbb{R}$ be given by
	\begin{equation*}
		\varphi_{0}(x) =\sqrt{1-|x|^2}.
	\end{equation*}
	Take $\psi \in C^{\infty}_c(\mathbb{R}^{N-1})$ to be a nonnegative compactly supported smooth function such that 
	\begin{equation*}
		\text{supp}(\psi) \subset B_{1}^{N-1},\, 0\leq \psi_t\leq 1\, \text{and } \psi_{\vert_{B_{1/2}^{N-1}}}\equiv 1.
	\end{equation*}
	Given $t>0$, we define $\psi_{t}(x):=\psi\left(\frac{x}{t}\right)$. Finally, for $0<t\leq t_1$ (where $t_1:=t_1(k,N)<1$ will be set later), we define $\varphi_{t}:B_1^{N-1}\rightarrow \mathbb{R}$ by setting
	\begin{equation}\label{varphi r0}
		\varphi_{t}(x)= \varphi_{0}(x)+\psi_{t}(x)\sum_{i=1}^{N-1}|x_i|^{k+\alpha}.
	\end{equation} 
	
	Let $\Omega_{t}\subset \mathbb{R}^{N}$ be the only bounded, open, connected set such that $\partial \Omega_{t} = \Gamma_{t}$. Note that, by construction, the regularity of $\Gamma_{t}$ matches the regularity of $\varphi_{t}$, which is clearly $C^{k,\alpha}$. Furthermore, the family of functions $\{\varphi_{t}\}$ is in $C^{k,\alpha}$ but not in $C^{k,\alpha'}$ for any $\alpha'>\alpha$.
    To simplify the presentation, we write $\Psi_{t}(x):= \varphi_{t}(x)-\varphi_{0}(x)$.
	
	To check the optimality of Theorem~\ref{thm:SBT stability} we need some estimates regarding $\Psi_t$, $|\nabla \Psi_t|$ and $|D^2\Psi_t|$.\\
	From now on, we assume that
	$$t_1\leq\frac{1}{10(N-1)(k+1)}.$$
    Note that since $\text{supp}\psi_t\subset B_{t}^{N-1}$ and $0\leq \psi_t\leq 1$, we have 
	\begin{equation}\label{sharp L infty bound}
			|\Psi_t(x)|\leq \sum_{i=1}^{N-1}|x_i|^{k+\al}\leq (N-1) t^{k+\alpha}
	\end{equation}
	Using the fact that $t\leq t_1$, we have the uniform (in $t$) bound
	\begin{equation}\label{L infty bound}
		|\Psi_t(x)|\leq \frac{1}{2}.
	\end{equation}
	Regarding $|\nabla \Psi_t|$ we have
	\begin{equation}\label{sharp gradient estimate}
		\begin{split}
			\left|\nabla \Psi_t(x)\right|&\leq \left|\psi_t(x)\nabla\left(\sum_{i =1}^{N-1}|x_i|^{k+\alpha}\right)+\left(\sum_{i =1}^{N-1}|x_i|^{k+\alpha}\right)\nabla\psi_{t}(x)\right|\\
			&\leq  (k+\alpha)\sum_{i=1}^{N-1}|x_i|^{k+\alpha-1}+\frac{\|\nabla \psi\|_{L^{\infty}(B_{1}^{N-1})}}{t} \sum_{i =1}^{N-1}|x_i|^{k+\alpha}\\
            &\leq (N-1)\left(k+\alpha + \|\nabla \psi\|_{L^{\infty}\left(B_{1}^{N-1}\right)}\right) t^{k+\alpha-1}.
		\end{split}
	\end{equation}
    Here, in the third inequality, we used the fact that $x\in B_{t}^{N-1}$.
	
	The same argument leads to
	\begin{equation}\label{sharp second order estimate}
		\left|D^2 \Psi_k(x)\right|\leq C_1 t^{k+\alpha-2},
	\end{equation}
	with $C_1$ being a positive constant dependent only on $N$, $k$, $\alpha$ and $\|\psi\|_{C^2(\mathbb{R}^{N-1})}$.
	
	From now on, whenever we make use of~\eqref{sharp gradient estimate} or~\eqref{sharp second order estimate} we will omit the dependence on $\|\psi\|_{C^2(\mathbb{R}^{N-1})}$ since it is a fixed quantity that does not change with any of the relevant parameters.\\
	Now let $z_t\in \Omega_{t}$ be such that 
	$$
	\inf_{z\in \Omega_{t}}\left(\rho_e(z)-\rho_i(z)\right) = \rho_e(z_t)-\rho_i(z_t).
	$$
	Simple geometric considerations imply that: 
	$$
	2\rho_e(z_t)\geq d_{\Omega_{t}}\geq 1+\sqrt{|x_t|^2+\psi_{t}(x_t)^2}\,\text{ and }\, \rho_i(z_t)\leq 1,
	$$
	where by $d_{\Omega_t}$ we mean the diameter of $\Omega_t$, and where $x_{t}\in B_{1}^{N-1}$ is such that $\max_{x\in B_{1}^{N-1}}\Psi_{t}(x) = \Psi_{t}(x_{t})$.
	So, the following chain holds:
	\begin{equation*}
		\begin{split}
			2\rho_e(z_t)-2\rho_i(z_t)&\geq 1+\sqrt{|x_t|^2+\psi_{t}(x_t)^2}-2= \sqrt{|x_t|^2+\psi_{t}(x_t)^2}-1\\
			&=\frac{|x_{t}|^2+\varphi_{t}(x_{t})^2 -1}{\sqrt{|x_{t}|^2+\varphi_{t}(x_{t})^2}+1}= \frac{2\varphi_{0}(x_{t})\Psi_{t}(x_{t})+\Psi_{t}(x_{t})^2}{\sqrt{1+2\varphi_{0}(x_{t})\Psi_{t}(x_{t})+\Psi_{t}(x_{t})^2}+1}\\
			&\geq \frac{2}{3} \Psi_{t}(x_{t}),
		\end{split}
	\end{equation*}
	where we used~\eqref{L infty bound} and $\Psi_{t}\geq 0$.
	Hence, recalling~\eqref{varphi r0} and that $\Psi_t = \varphi_t-\varphi_0$, we obtain
	\begin{equation}\label{difference of the radii}
		\rho_e(z_t)-\rho_i(z_t)\geq \frac{2}{3}\frac{t^{k+\alpha}}{2^{k+\alpha}}.
	\end{equation}

	Now, all that is left for us to do is study the behaviour of the $L^2$-deviation of the mean curvature of the sets $\Omega_{t}$ from the reference constant $H_0:= \frac{|\Gamma_t|}{N|\Omega_{t}|}$ as a function of $t$.
	
	To do this, we break $\|H_{t}-H_0\|_{L^r(\Gamma_{t})}$ into three terms:
	\begin{equation}\label{upper bound for mean curvature}
		\begin{split}
			\|H_{t}-H_0\|_{L^r(\Gamma_{t})}&\leq \|H_{t}-1\|_{L^r(\Gamma_{t})} + |\Gamma_{t}|^{\frac{1}{r}}|H_{0}-1|\\
			&\leq \|H_{t}-1\|_{L^r(\Gamma_{t})} + |\Gamma_{t}|^{\frac{1}{r}}\frac{\left||\Gamma_{t}|-\left|\mathbb{S}^{N-1}\right|\right|}{N|\Omega_{t}|}+\frac{|\Gamma_{t}|^{\frac{1}{r}}\left|\mathbb{S}^{N-1}\right|}{N|B_1||\Omega_{t}|}||\Omega_{t}|-|B_1||\\
			&\leq \|H_{t}-1\|_{L^r(\Gamma_{t})} + C_2||\Gamma_{t}|-\left|\mathbb{S}^{N-1}\right||+C_2||\Omega_{t}|-|B_1||,\\
		\end{split}
	\end{equation}
	where $C_2$ is a positive constant dependent on $N$ and $r$, and where we have used the fact that because $t\leq t_{1}(k,N)\leq \frac{1}{10(N-1)(k+1)}$, we have $|\Omega_{t}|\geq \frac{|B_1|}{2}$ and $|\Gamma_{t}|\leq 2|\mathbb{S}^{N-1}|$.
	
	Now, we estimate $\|H_{t}-1\|_{L^r(\Gamma_{t})}$, $\left||\Gamma_{t}|-\left|\mathbb{S}^{N-1}\right|\right|$ and $||\Omega_{t}|-|B_1||$, individually.
	
	First, for the deviation of the measures, using~\eqref{sharp L infty bound} we see that there exists a positive constant $C_9$ dependent only on $N$, such that:
	\begin{equation}\label{Measures}
			\left||\Omega_{t}|-|B_1|\right|=\left|\int_{B_1^{N-1}}\left(\varphi_{t}-\varphi_0\right)dx\right|\leq \int_{B_{t}^{N-1}}|\Psi_{t}(x)|dx \leq \left|\SS^{N-2}\right|\,t^{k+\alpha+N-1}.
	\end{equation}
	
	Second, the deviation of the perimeters
	\begin{equation}\label{Perimeters}
		\begin{split}
			\left||\Gamma_{t}|-\left|\mathbb{S}^{N-1}\right|\right|&=\left|\int_{B_1^{N-1}}\left(\sqrt{1+|\nabla \varphi_{t}|^2}-\sqrt{1+|\nabla\varphi_{0}|^2}\right)\,dx\right|\\
			&=\left|\int_{B_1^{N-1}}\frac{|\nabla \varphi_{t}|^2-|\nabla \varphi_{0}|^2}{\sqrt{1+|\nabla \varphi_{t}|^2}+\sqrt{1+|\nabla \varphi_{0}|^2}}\,dx\right|\leq\int_{B_{1}^{N-1}}\left||\nabla \varphi_{t}|^2-|\nabla \varphi_{0}|^2\right|\,dx\\
			& = \int_{B_1^{N-1}}\left||\nabla \Psi_{t}|^2+2\langle \nabla \Psi_{t},\nabla \varphi_{0}\rangle\right|\,dx\leq \int_{B_{t}^{N-1}}|\nabla \Psi_{t}|^2+2|\nabla \Psi_{t}||\nabla \varphi_{0}|\,dx\\
			&\leq C_{3} t^{k+\alpha+N-1},
		\end{split}
	\end{equation}
	where $C_{3}$ is a positive constant dependent only on $N$, and where we have used~\eqref{sharp gradient estimate} and the fact that $\left|\nabla\varphi_0(x) \right|\leq 2t$ on $B_{t}^{N-1}$.
	
	Before estimating $\|H_{t}-1\|_{L^r(\Gamma_{t})}$ we first give a pointwise estimate on $|H_{t}-1|$.
	For that, note that, by construction, we have $H_{t}(x)=1$ for all $x\in \Gamma_{t}\setminus \left(B_{t}^{N-1}\times (0,+\infty)\right)$. Therefore, to estimate the deviation of the mean curvature from $1$ we only need to look at the mean curvature of the portion of $\Gamma_t$ inside $B_{t}^{N-1}\times (0,+\infty)$. To that end, recall that for the sets under consideration, the mean curvature of the portion of $\Gamma_t$ inside $B_{t}^{N-1}\times (0,+\infty)$ is given by the formula (see, for instance, \cite[Chapter 14]{GT}): 
	\begin{equation*}
		H_t(x,\varphi_{t}(x))=-\frac{\Delta \varphi_{t}(x)}{(N-1)(1+|\nabla \varphi_{t}(x)|^2)^{1/2}}+\frac{\left(\nabla \varphi_{t}(x)\right)^T D^2 \varphi_{t}(x)\nabla \varphi_{t}(x) }{(N-1)(1+|\nabla \varphi_{t}(x)|^2)^{3/2}}.
	\end{equation*}

	Expanding the formula above for the mean curvature, we have (we omit the dependence on the $x-$variable to simplify the presentation)
	
	\begin{equation}\label{Pointwise deviation of the mean curvature}
		\begin{split}
			&(N-1)\left(H_{t}(x,\varphi_{t}(x))-1\right)= -\frac{\Delta \varphi_{0} + \Delta \Psi_{t}}{(1+|D \varphi_{t}|^2)^{1/2}}+\frac{\left(\nabla \varphi_{0}+\nabla \Psi_{t}\right)^T \left(D^2\varphi_{0}+D^2\Psi_{t}\right)\left(\nabla \varphi_{0}+\nabla\Psi_{t}\right) }{(1+|\nabla \varphi_{t}|^2)^{3/2}}\\
			&+\frac{\Delta \varphi_{0}}{(1+|\nabla \varphi_{0}|^2)^{1/2}}-\frac{\nabla\varphi_{0}^T D^2 \varphi_{0} \nabla \varphi_{0}}{(1+|\nabla \varphi_{0}|^2)^{3/2}}\\ 
			&= -\frac{\Delta \varphi_{0} +\Delta \Psi_{t}}{(1+|\nabla \varphi_{t}|^2)^{1/2}}+\frac{\nabla \varphi_{0}^T D^2\varphi_{0}\nabla \varphi_{0}}{(1+|\nabla \varphi_{t}|^2)^{3/2}}+\frac{\nabla \varphi_{0}^T D^2\varphi_{0}\nabla \Psi_{t}}{(1+|\nabla \varphi_{t}|^2)^{3/2}}+\frac{\nabla \varphi_{0}^T D^2\Psi_{t} \nabla\varphi_{0} }{(1+|\nabla \varphi_{t}|^2)^{3/2}}\\
			&+\frac{\nabla\varphi_{0}^T D^2\Psi_{t}\nabla\Psi_{t}}{(1+|\nabla\varphi_{t}|^2)^{3/2}}+\frac{\nabla\Psi_{t}^T D^2\varphi_{0} \nabla\varphi_{0}}{(1+|\nabla \varphi_{t}|^2)^{3/2}}+\frac{\nabla \Psi_{t}^T D^2\varphi_{0}\nabla \Psi_{t}}{(1+|\nabla \varphi_{t}|^2)^{3/2}}+\frac{\nabla \Psi_{t}^T D^2\Psi_{t}\nabla \varphi_{0}}{(1+|\nabla \varphi_{t}|^2)^{3/2}}\\
			&+\frac{\nabla \Psi_{t}^T D^2\Psi_{t}\nabla\Psi_{t}}{(1+|\nabla \varphi_{t}|^2)^{3/2}}+\frac{\Delta \varphi_{0}}{(1+|\nabla \varphi_{0}|^2)^{1/2}}-\frac{\nabla \varphi_{0}^T D^2 \varphi_{0} \nabla\varphi_{0} }{(1+|\nabla \varphi_{0}|^2)^{3/2}}\\
			&=\left((1+|\nabla \varphi_{0}|^2)^{-1/2}-(1+|\nabla \varphi_{t}|^2)^{-1/2} \right)\Delta \varphi_{0}\\
			&+\left((1+|\nabla \varphi_{t}|^2)^{-\frac{3}{2}}-(1+|\nabla \varphi_{0}|^2)^{-\frac{3}{2}}\right)\nabla \varphi_{0}^T D^2\varphi_{0}\nabla \varphi_{0} \\
			&-\frac{\Delta \Psi_{t}}{(1+|\nabla \varphi_{t}|^2)^{1/2}} + 2\frac{\nabla \varphi_{0}^T D^2\varphi_{0}\nabla \Psi_{t}}{(1+|\nabla \varphi_{t}|^2)^{3/2}}+\frac{\nabla \varphi_{0}^T D^2\Psi_{t}\nabla \varphi_{0}}{(1+|\nabla \varphi_{t}|^2)^{3/2}}\\
			&+2\frac{\nabla \varphi_{0}^T D^2\Psi_{t}\nabla \Psi_{t}}{(1+|\nabla \varphi_{t}|^2)^{3/2}}+\frac{\nabla \Psi_{t}^T D^2\varphi_{0}\nabla \Psi_{t}}{(1+|\nabla \varphi_{t}|^2)^{3/2}}+\frac{\nabla \Psi_{t}^T D^2\Psi_{t}\nabla \Psi_{t}}{(1+|\nabla \varphi_{t}|^2)^{3/2}}.\\
		\end{split}
	\end{equation}
	
	Note that for $j = 1,3$, we have
	\begin{equation}\label{Est on roots}
		\begin{split}
			\left|(1+|\nabla \varphi_{t}|^2)^{-j/2}-(1+|\nabla \varphi_{0}|^2)^{-j/2}\right|&= \left|\int_{|\nabla \varphi_{t}|}^{|\nabla\varphi_{0}|}\frac{j s}{(1+s^2)^{\frac{j}{2}+1}}\,ds\right|\leq 2  \left||\nabla \varphi_{0}|^2-|\nabla \varphi_{t}|^2\right|\\
			&= 2\left||\nabla \Psi_{t}|^2+ 2\langle \nabla \Psi_{t},\nabla \varphi_0\rangle\right|\leq 2|\nabla \Psi_{t}|^2 + 4|\nabla \Psi_{t}||\nabla \varphi_{0}|.
		\end{split}
	\end{equation}
	Recall the formulas for $\nabla \varphi_{0}$ and $D^2 \varphi_{0}$:
	\begin{equation*}
		\nabla \varphi_{0}(x) = -\frac{x}{\sqrt{1-|x|^2}},
	\end{equation*}
	and 
	\begin{equation*}
		\left(D^2 \varphi_{0} (x)\right)_{ij} = -\frac{1}{\sqrt{1-|x|^2}}\delta_{ij} +\frac{1}{(1-|x|^2)^{3/2}}x_i x_j = -\frac{1}{\varphi_0(x)}\delta_{ij} +\frac{1}{\varphi_{0}(x)^3}x_i x_j,
	\end{equation*}
	We see that for $t\leq t_1$,
	\begin{equation}\label{Eq: First and Second Order est varphi_0}
		|\nabla \varphi_{0}(x)|\leq 2 |x|,\,\text{and }	|D^2 \varphi_{0} (x)| \leq 3,\quad \text{for all }x\in B_{t}^{N-1}.
	\end{equation}
	
	From~\eqref{Pointwise deviation of the mean curvature} we obtain 
	\begin{equation}\label{Pointwise deviation of mean curvature -1}
		\begin{split}
			\left|(N-1)(H_{t}-1)+\frac{\Delta \Psi_{t}}{(1+|\nabla \varphi_{t}|^2)^{1/2}}\right|&\leq \left|(1+|\nabla \varphi_{0}|^2)^{-1/2}-(1+|\nabla \varphi_{t}|^2)^{-1/2} \right|\left|\Delta \varphi_{0}\right|\\
			&+\left|(1+|\nabla \varphi_{t}|^2)^{-\frac{3}{2}}-(1+|\nabla \varphi_{0}|^2)^{-\frac{3}{2}}\right|\left|\nabla \varphi_{0}^T D^2\varphi_{0}\nabla \varphi_{0}\right|\\
			&+ 2\left|\frac{\nabla \varphi_{0}^T D^2\varphi_{0}\nabla \Psi_{t}}{(1+|\nabla \varphi_{t}|^2)^{3/2}}\right|+\left|\frac{\nabla \varphi_{0}^T D^2\Psi_{t}\nabla \varphi_{0}}{(1+|\nabla \varphi_{t}|^2)^{3/2}}\right|\\
			&+2\left|\frac{\nabla \varphi_{0}^T D^2\Psi_{t}\nabla \Psi_{t}}{N(1+|\nabla \varphi_{t}|^2)^{3/2}}\right|+\left|\frac{\nabla \Psi_{t}^T D^2\varphi_{0}\nabla \Psi_{t}}{(1+|\nabla \varphi_{t}|^2)^{3/2}}\right|+\left|\frac{\nabla \Psi_{t}^T D^2\Psi_{t}\nabla \Psi_{t}}{(1+|\nabla \varphi_{t}|^2)^{3/2}}\right|\\
			&\leq 2 \left(|\nabla \Psi_{t}|^2 + |\nabla \Psi_{t}||\nabla \varphi_{0}| \right)\left(|\Delta \varphi_{0}|+|\nabla \varphi_{0}|^2 |D^2\varphi_{0}|\right)\\
			&+|\nabla \varphi_{0}|| D^2\varphi_{0}||\nabla \Psi_{t}|+|\nabla \varphi_{0}|^2 |D^2\Psi_{t}|+|\nabla \varphi_{0}| |D^2\Psi_{t}||\nabla \Psi_{t}|\\
			& +|\nabla \Psi_{t}|^2 |D^2\varphi_{0}|+|\nabla\Psi_{t}|^2 |D^2\Psi_{t}|\\
			&\leq C_{4}\left(\left(t^{2(k+\alpha))-2} + t^{k+\alpha} \right)\left(1+t^2\right)+t^{k+\alpha}+ t^{k+\alpha+1}\right)\\
			&+ C_{4}\left(t^{3(k+\alpha))-2}+ t^{2(k+\alpha))-2}+t^{3(k+\alpha))-4}\right)\\
			&\leq 9C_{4} t^{k+\alpha},
		\end{split}
	\end{equation}
	where $C_{4}$ is a positive constant dependent on $N$, $k$ and $\alpha$. Here, in the first inequality, we used~\eqref{Pointwise deviation of the mean curvature}. In the second inequality, we used~\eqref{Est on roots}. In the third inequality, we used~\eqref{Eq: First and Second Order est varphi_0},~\eqref{sharp gradient estimate} and~\eqref{sharp second order estimate} (note that $x\in B_{t}^{N-1}$). Finally, for the fourth inequality, we used $t<1$ together with $k\geq 2$.
	So, there is a positive constant $C_{5}$ dependent only on $N$, $k$ and $\alpha$ such that 
	\begin{equation*}
		\left|H_{t}-1\right|\leq \left|\Delta \Psi_{t}\right|+ C_{5}t^{k+\alpha}.
	\end{equation*}
	
	Using this in combination with~\eqref{sharp gradient estimate},~\eqref{sharp second order estimate} and recalling that $t\leq t_1$, we see that 
	\begin{equation*}
		\begin{split}
			\|H_{t}-1\|_{L^r(\Gamma_{t})}^r&\leq \int_{B_{t}^{N-1}} \left(\left|\Delta \Psi_{t}\right|+C_{5}t^{k+\alpha}\right)^r\sqrt{1+|\nabla\varphi_{t}(x)|^2}\,dx\\
			&\leq C_{6}\left(1+ t^2\right)^r t^{r(k+\alpha-2)+N-1}\\
			&\leq 2C_{6}\, t^{r(k+\alpha)-2r+N-1},
		\end{split}
	\end{equation*}
	for some positive constant $C_{6}$ dependent only on $N$.

	Combining this with~\eqref{upper bound for mean curvature},~\eqref{Measures} and~\eqref{Perimeters}, we conclude (for $t\leq t_1$) that 
	\begin{equation}\label{L2 deviation of the mean curvature one sided}
		\|H_{t}-H_0\|_{L^r(\Gamma_{t})}\leq  C_{7} t^{k+\alpha+\frac{N-1-2r}{r}},
	\end{equation}
	for some positive constant $C_{7}$ that depend solely on $N$, $k$, $\alpha$ and $r$.
	
	To check the optimality of Theorem~\ref{thm:SBT stability} we combine~\eqref{difference of the radii} and~\eqref{L2 deviation of the mean curvature one sided} to obtain
	\begin{equation*}
		\|H_{t}-H_0\|_{L^r(\Gamma_{t})}^{\tau_{k,\alpha,N,r}}\leq C_{7}^{\tau_{k,\alpha,N,r}} t^{\left(k+\alpha+\frac{N-1-2r}{r}\right)\tau_{k,\alpha,N,r}} \leq C_{7}^{\tau_{k,\alpha,N,r}} (\rho_e(t)-\rho_i(t)),
	\end{equation*}
	where $\tau_{k,\alpha,N,r}$ is given by~\eqref{eq:stability exponent I}.
	Thus, the optimality of Theorem~\ref{thm:SBT stability} is confirmed.
	\begin{rem}
    {\rm
		We point out that, using~\eqref{Pointwise deviation of mean curvature -1}, one can also find a constant $C_{8}$ dependent only on $N$, $k$, $\alpha$ and $p$, such that 
		$$
		\|H_{t}-1\|_{L^r(\Gamma_{t})}\geq \frac{C_{8}}{2^{k+\alpha}}t^{k+\alpha+\frac{N-1-2r}{r}},
		$$
		for $t\leq t_1$, provided that $t_1$ is sufficiently small.
		
		Then, noting that, by the inverse triangle inequality, we also have 
		$$
		\|H_{t}-H_0\|_{L^r(\Gamma_{t})}\geq \|H_{t}-1\|_{L^r(\Gamma_{t})}-|H_{0}-1|\geq \|H_{t}-1\|_{L^r(\Gamma_{t})} -C_2||\Gamma_{t}|-\left|\mathbb{S}^{N-1}\right||-C_2||\Omega_{t}|-|B_1||,
		$$
		where $C_2$ is the constant in~\eqref{upper bound for mean curvature}.
		We also see (for $t\leq t_1$) that
		$$
		\|H_{t}-H_0\|_{L^r(\Gamma_{t})}\geq \frac{C_9}{2^{k+\alpha}}t^{k+\alpha+\frac{N-1-2r}{r}},
		$$
		for some constant $C_{9}$ dependent only on $N$, $k$, $\alpha$ and $r$. This shows that the asymptotic behaviour of $\|H_{t}-H_0\|_{L^r(\Gamma_{t})}$, as $t\to 0^+$, is the same (up to constants dependent on $N$, $k$, $\alpha$ and $r$) as $t^{k+\alpha+\frac{N-1-2r}{r}}$.
    }
	\end{rem}

    \appendix
    \counterwithin*{equation}{section}
    \renewcommand\theequation{\thesection.\arabic{equation}}
    
	\section{Optimality of Theorem~\ref{thm:SBT stability}:
    The $C^{1,\alpha}$ case}\label{sec:Optimality in the rough case}
	
	We now show how to adapt the estimates in Section~\ref{Sec: Optimality} to show the optimality of Theorem~\ref{thm:SBT stability} in the case where the domain is of class $C^{1,\alpha}$ with $\alpha \in \left(0,1\right)$. In fact, we will take $\alpha \in \left(\frac{r-1}{r},1\right)$ for a (technical) reason that will be made clear later.
	
	To do this, we make the same construction as in Section~\ref{Sec: Optimality}, where in~\eqref{varphi r0} we set $k=1$.\newline
	Note that, as in Section~\ref{Sec: Optimality}, for $t<1$, the $C^{1,\alpha}$ regularity of the family of sets $\{\Gamma_t\}$ is uniformly bounded. Furthermore, for $t<t_1$ (where $t_1:=t_{1}(N)\leq \frac{1}{10(N-1)}$),~\eqref{sharp L infty bound},~\eqref{L infty bound} and~\eqref{sharp gradient estimate} still hold. Consequently,~\eqref{difference of the radii} and~\eqref{Measures} remain valid. However,~\eqref{Perimeters} needs to be updated as follows:
	\begin{equation}\label{Perimeters rough}
		\begin{split}
			\left||\Gamma_{t}|-\left|\mathbb{S}^{N-1}\right|\right|&\leq \int_{B_{t}^{N-1}}|\nabla \Psi_{t}|^2+2|\nabla \Psi_{t}||\nabla \varphi_{0}|\,dx\\
			&\leq C_{1} t^{2\alpha+N-1},
		\end{split}
	\end{equation}
	where $C_{1}$ is a positive constant dependent on $N$.
	Here, in the first inequality, we proceeded as in the first three lines of~\eqref{Perimeters} and for the second inequality, we used~\eqref{sharp gradient estimate} (with $k = 1$) together with $\alpha<1$.
	
	Another difference that one encounters when adapting the estimates in Section~\ref{Sec: Optimality} to the $C^{1,\alpha}$ case is when we need to take into account second-order terms. This is because, when taking two derivatives of $\Psi_t$, singularities appear. For this reason, to avoid the points where singularities appear, we will do our (pointwise) estimates for $x\in \left\{x\in B_{4t}^{N-1}:\forall i\in \{1,..., N-1\}, x_i\neq 0\right\}$.
	Indeed, when estimating a pointwise upper-bound for $|D^2\Psi_t(x)|$ ($x\in\left\{x\in B_{4t}^{N-1}:\forall i\in \{1,...,N-1\},\,x_i\neq 0\right\}$), we now obtain:
	\begin{equation}\label{sharp second order est rough case}
			|D^2\Psi_t(x)|\leq C_{2}\sum_{i=1}^{N-1}|x_i|^{\alpha-1},
	\end{equation}
	where $C_{2}$ is a positive constant dependent only on $N$.\\
	By proceeding just as in Section~\ref{Sec: Optimality} we obtain:
	\begin{equation*}
		\begin{split}
			\left|(N-1)(H_{t}-1)+\frac{\Delta \Psi_{t}}{(1+|\nabla \varphi_{t}|^2)^{1/2}}\right|&\leq 2 \left(|\nabla \Psi_{t}|^2 + |\nabla \Psi_{t}||\nabla \varphi_{0}| \right)\left(|\Delta \varphi_{0}|+|\nabla \varphi_{0}|^2 |D^2\varphi_{0}|\right)\\
			&+|\nabla \varphi_{0}|| D^2\varphi_{0}||\nabla \Psi_{t}|+|\nabla \varphi_{0}|^2 |D^2\Psi_{t}|+|\nabla \varphi_{0}| |D^2\Psi_{t}||\nabla \Psi_{t}|\\
			& +|\nabla \Psi_{t}|^2 |D^2\varphi_{0}|+|\nabla\Psi_{t}|^2 |D^2\Psi_{t}|\\
			&=  2 \left(|\nabla \Psi_{t}|^2 + |\nabla \Psi_{t}||\nabla \varphi_{0}| \right)\left(|\Delta \varphi_{0}|+|\nabla \varphi_{0}|^2 |D^2\varphi_{0}|\right)\\
			&+|\nabla \varphi_{0}|| D^2\varphi_{0}||\nabla \Psi_{t}|+|\nabla \Psi_{t}|^2 |D^2\varphi_{0}|\\
			& +\left(|\nabla \varphi_{0}|^2 +|\nabla \varphi_{0}||\nabla \Psi_{t}|+ |\nabla\Psi_{t}|^2\right)\left|D^2\Psi_{t}\right|\\
			&\leq C_{3}\left(t^{2\alpha}+ t^{1+\alpha}\right)(1+t^2)+ C_{3} t^{1+\alpha} + C_{3}t^{2\alpha}\\
			&+C_{3}\left(t^2 + t^{1+\alpha}+t^{2\alpha}\right)\sum_{i=1}^{N-1}|x_i|^{\alpha-1}\\
			&\leq C_{4}t^{2\alpha}\sum_{i=1}^{N-1}|x_i|^{\alpha-1},
		\end{split}
	\end{equation*}
	where $C_{3}$ and $C_{4}$ are positive constants dependent only on $N$. In the chain of estimates above, the first inequality follows from~\eqref{Pointwise deviation of the mean curvature} and the first two inequalities in~\eqref{Pointwise deviation of mean curvature -1}, the second inequality follows from~\eqref{sharp gradient estimate},~\eqref{Eq: First and Second Order est varphi_0} and~\eqref{sharp second order est rough case}, and the third inequality is a consequence of the facts that $t\leq t_1$ and $\alpha<1$.
	
	Thus, we see that
	\begin{equation*}
		\left|H_{t}-1\right|\leq \left|\Delta \Psi_{t}\right|+ C_{4} t^{2\alpha}\sum_{i=1}^{N-1}|x_i|^{\alpha-1}.
	\end{equation*}
	Noting that by construction, we have 
	\begin{equation*}
		\Delta \Psi_t = (1+\alpha)\alpha \sum_{i=1}^{N-1}\frac{|x_i|^{\alpha-1}}{4^{1+\alpha}}.
	\end{equation*}
	This leads to the following pointwise estimate for the deviation of the mean curvature: 
	\begin{equation*}
		\left|H_{t}-1\right|\leq C_{5} \sum_{i=1}^{N-1}|x_i|^{\alpha-1},
	\end{equation*}
	where $C_{5}$ is a positive constant dependent only on $N$.
	So, we have
	\begin{equation}\label{L2 Mean curvature deviation rough case}
		\begin{split}
			\|H_t-1\|_{L^r(\Gamma_t)}^{r}&= \int_{B_{t}^{N-1}} \left|H_t-1\right|^r\sqrt{1+|\nabla\varphi_t|^2}\,dx\\
			&= \int_{\left\{x\in B_{t}^{N-1}:\forall i\in \{1,..., N-1\}, x_i\neq 0\right\}} \left|H_t-1\right|^r\sqrt{1+|\nabla\varphi_t|^2}\,dx\\
			&\leq C_{6}  \int_{\left\{x\in B_{t}^{N-1}:\forall i\in \{1,..., N-1\}, x_i\neq 0\right\}} \left(\sum_{i=1}^{N-1}|x_i|^{\alpha-1}\right)^r\,dx\\
			&\leq C_{6}\int_{-t}^{t}\cdots \int_{-t}^{t} \left(\sum_{i=1}^{N-1}|x_i|^{\alpha-1}\right)^r\,dx_1\cdots dx_{N-1}\\
			&=C_{6}2^{N-1} \int_{0}^{t}\cdots \int_{0}^{t}\left(\sum_{i=1}^{N-1}x_{i}^{(\alpha-1)r}\right)dx_1\cdots dx_{N-1}\\
			&\leq \frac{C_{7}}{r\alpha+1-r} t^{(1+\alpha)r+N-1-2r}
		\end{split}
	\end{equation}
	where $C_{6}$ and $C_{7}$ are positive constants dependent only on $N$. On the second line, we used the fact that $\left|B_{t}^{N-1}\setminus\left\{x\in B_{t}^{N-1}:\forall i\in \{1,..., N-1\}, x_i\neq 0\right\}\right|=0$. On the last inequality, we used the assumption that $\alpha>\frac{r-1}{r}$ (without this assumption, note that $H_t\notin L^r (\Gamma_t)$). 
	
	Also note that combining~\eqref{sharp second order est rough case} and the last four lines in~\eqref{L2 Mean curvature deviation rough case} show that $D^2 \Psi_t \in L^r(B_{t}^{N-1})$, with $\|D^2 \Psi_t\|_{L^r(B_{t}^{N-1})}$ being uniformly bounded with respect to $t<t_1$. As a consequence, we have that the family of sets $\{\Gamma_t\}$ is in $W^{2,r}$. 
	
	Now, to conclude that the estimates of Theorem~\ref{thm:SBT stability} for $C^{1,\alpha}$ domains are optimal, we proceed like we did in Section~\ref{Sec: Optimality} to confirm the optimality of Theorem~\ref{thm:SBT stability} for more regular domains.

	\section{Proof of Lemma \ref{lem: Generalization of 1.2}}\label{Proof of lemma of generalized 1.2}
    Throughout this section, let $r\in (1,\infty)$ and $\alpha\in(0,1]$ and $\Omega\subset \mathbb{R}^N$ be a bounded domain whose boundary $\Ga$ is of class $C^{1,\alpha}\cap W^{2,r}$.
    
    To prove Lemma \ref{lem: Generalization of 1.2}, we consider a solution, $u$, to the so-called torsion problem, that is, the following boundary value problem:
	\begin{equation}\label{eq:torsionproblem}
	\begin{cases}
		\De u = N \quad & \text{ in } \Om,
		\\
		u=0 \quad & \text{ on } \Ga.
	\end{cases}
	\end{equation}
    Since $u$ is the solution to \eqref{eq:torsionproblem}, by standard elliptic regularity (see, e.g., \cite{GT}) we have that $u\in C^{1,\al}(\ol{\Om})$ and
    \begin{equation}\label{eq:elliptic regularity u Schauder}
		\nr u \nr_{C^{1,\al}(\ol{\Om})} \le C,
	\end{equation}
	for a constant $C$ only depending on $N,\al$, and the $C^{1,\al}$-regularity of $\Ga$.
    Given a point $z\in \RR^N$, for $x \in \ol{\Om}$ we set
	\begin{equation}\label{eq:def h and q}
		h_z(x) := q_z(x) - u(x) , \quad \text{ where } \quad q_z(x) := \frac{|x-z|^2}{2}  .
	\end{equation}

    The following three lemmas are essentially contained in \cite{MP2}, except for the fact that the uniform sphere condition used there is replaced by the $C^{1,\al}$-regularity of $\Ga$ (with $0<\al\le 1$) here. This requires some technical modifications that we collect here for the reader's convenience. This is in analogy with what was done in \cite[Appendix]{CPY} for the stability results obtained in \cite{MP2,MP3,MP6} concerning Serrin's overdetermined problem.

	We are working with the $C^{1,\al}$-regularity of $\Ga$; nevertheless, we mention that in some results such a parameter may be weakened by exploiting the notion of interior pseudo-ball condition in \cite{ABMMZ}.
	
	The first lemma adapts to our setting \cite[Lemma~3.1]{MP2}, whereas the second lemma adapts to our setting \cite[Lemma~3.4]{MP2}. The third lemma provides an explicit upper bound for $|\Ga|$, in terms of $N$, $|\Om|$ and the $C^{1,\al}$ regularity of $\Ga$; this will be useful to simplify the dependencies of the constants in our estimates.
	
	
	
	\begin{lem}\label{lem:distance growth}
		Given $N\ge2$, let $\Om\subset \RR^N$ be a bounded domain in $\RR^N$ with boundary $\Ga$ of class $C^{1,\al}$ and let $u$ 
		be the solution to \eqref{eq:torsionproblem}.
		
		Then,
		\begin{equation}\label{eq:distanceLipschitzgrowth}
		- u(x) \ge C \, \de_{\Ga}(x) \qquad \qquad \text{for any } x\in\ol{\Om} ,
		\end{equation}
		where $\de_{\Ga}(x)$ is the distance function from $\Ga$, and where $C$ is a positive constant dependent only on $N$ and the $C^{1,\al}$-regularity of $\Om$.
	\end{lem}
	\begin{proof}
	By comparison, we easily obtain the following rough estimate
	\begin{equation}\label{eq:rough estimate distance2}
		- u(x) \ge \frac{\de_{\Ga}(x)^2}{2} \qquad \qquad \text{for any } x\in\ol{\Om}.
	\end{equation}
	The last inequality easily follows by comparing $u$ with $w(y):= \left( |y-x|^2 - \de_{\Ga}(x)^2 \right)/2$ (that is the solution of \eqref{eq:torsionproblem} in the ball $B_{\de_{\Ga}(x)}(x)$ centered at $x$ with radius $\de_{\Ga}(x)$). We stress that we have not used the $C^{1,\al}$-regularity of $\Ga$ yet, and in fact \eqref{eq:rough estimate distance2} remains true without such an assumption (see, e.g, \cite[Lemma~3.1]{MP2}).
	
	By the Hopf-Olenik lemma for $C^{1,\al}$ domains (see, e.g., \cite{Gi} and \cite[Theorem~4.4]{ABMMZ}), for any $y\in\Ga$ we have that
	\begin{equation}\label{eq:HopfC1al}
	- u(y - t \, \nu(y))> \kappa t \qquad \qquad \text{for any } 0<t< \ul{\de} ,
	\end{equation}
	where $\kappa$ and $\ul{\de}$ are two constants only depending on $N$ and the $C^{1,\al}$-regularity of $\Ga$. Here, as usual, $\nu(x)$ denotes the outward unit normal to $\Ga$ at $x$. 
	Let $x$ be any point in $\ol{\Om}$.
	If $\de_{\Ga}(x) < \ul{\de}$, then \eqref{eq:HopfC1al} with $t:=\de_{\Ga}(x)$ and $y\in\Ga$ such that $|y-x|=\de_{\Ga}(x)$ (so that $y - t \, \nu(y) = x $) gives that \eqref{eq:distanceLipschitzgrowth} holds true with $C:=\kappa$.
	On the other hand, if $\de_{\Ga}(x) \ge \ul{\de}$, \eqref{eq:rough estimate distance2} easily gives \eqref{eq:distanceLipschitzgrowth} with $:=\ul{\de}/(2N)$.
	
	In any case, for any $x \in \ol{\Om}$, \eqref{eq:distanceLipschitzgrowth} always holds true with $C:=\max \left\lbrace \kappa ,\ul{\de}/(2N) \right\rbrace$.
	\end{proof}
	
	\begin{lem}\label{lem: lemma 2 in appendix for C 1 alpha}
	Given $N\ge2$, let $\Om\subset \RR^N$ be a bounded domain in $\RR^N$ with boundary $\Ga$ of class $C^{2}$. Let $z$ be the center of mass of $\Om$, that is,
    \begin{equation*}
        z= \frac{1}{|\Om|} \int_\Om x \, dx .
    \end{equation*} \
    Let $u$ be the solution to \eqref{eq:torsionproblem} and $h_z$ be as defined in \eqref{eq:def h and q}.
	Then, we have that
	\begin{equation*}
	\int_{\Ga} | \na h_z |^2 dS_x \le C \int_\Om (-u) | D^2 h_z|^2 dx,
	\end{equation*}
	where $C$ is a constant only depending on $N$, the diameter of $\Om$, and the $C^{1,\al}$-regularity of $\Ga$.
	\end{lem}
	\begin{proof}
    Note that by the choice of $z$ and the definition of $h_z$, it follows that
    $$
        \int_\Om \na h_z \, dx = 0.
    $$
	The result easily follows using the same argument of the proof of \cite[(ii) of Lemma 2.5]{MP3} (with $v:=h_z$), but replacing \cite[(1.15)]{MP3} with \eqref{eq:distanceLipschitzgrowth} and \cite[Theorem 3,10]{MP} with
	\begin{equation*}
		u_\nu \ge \kappa , \quad \text{ where } \kappa \text{ is that appearing in \eqref{eq:HopfC1al},}
	\end{equation*}
	which immediately follows from \eqref{eq:HopfC1al}.
	\end{proof}

    \begin{lem}\label{lem:bound on surface measure Ga}
        Given $N\ge 2$, let $\Om \subset \RR^N$ with boundary $\Ga$ of class $C^{1,\alpha}$. Then, we have that 
        \begin{equation*}
            |\Ga| \leq \frac{N|\Om|}{\kappa},
        \end{equation*}
        where $\kappa$ is the constant appearing in \eqref{eq:HopfC1al}.
    \end{lem}
    \begin{proof}
        Let $u\in C^{1,\alpha}(\overline{\Om})$ be the solution of \eqref{eq:torsionproblem}. 
        Using the divergence theorem, we obtain the following identity
        $$
        N|\Om| = \int_{\Ga} u_{\nu} \nu \, d\mathcal{H}^{N-1}.
        $$
        Combining this with the fact that $u_{\nu} \geq \kappa$, which follows immediately from \eqref{eq:HopfC1al}, we obtain 
        $$
        N|\Om| \geq \kappa |\Ga|,
        $$
        from which the result follows immediately.
    \end{proof}

    \begin{proof}[Proof of Lemma \ref{lem: Generalization of 1.2}]
    Start by assuming that $\Ga$ is of class $C^2$. Recall the following fundamental identity, which was proven in \cite{MP}
        \begin{equation*}
            \frac{1}{N-1}\int_{\Omega}|D^2 h_z|^2\,dx + H_{0}\int_{\Ga}\left(u_{\nu}-\frac{1}{H_0}\right)^2\,d\mathcal{H}^{N-1} = \int_{\Ga} \left(H-H_0\right)(u_{\nu})^2\,d\mathcal{H}^{N-1}.
            \end{equation*}
    Combining this identity,~\eqref{eq:elliptic regularity u Schauder} and Hölder's inequality, we find that 
    \begin{equation}\label{eq:Holder on fundamental identity}
        \frac{1}{N-1}\int_{\Omega}|D^2 h_z|^2\,dx + H_{0}\int_{\Ga}\left(u_{\nu}-\frac{1}{H_0}\right)^2\,d\mathcal{H}^{N-1}\leq C \|H-H_0\|_{L^r(\Ga)},
    \end{equation}
    where $C$ is a positive constant dependent on $r$, $N$, $|\Ga|$ and the $C^{1,\alpha}$ regularity of $\Ga$.

    Using the triangle inequality, we see that 
    \begin{equation}\label{eq:estimate on term 1 in generalized 1.2}
         \left\||x-z|-\frac{1}{H_0}\right\|_{L^2(\Ga)}\leq \left\|\frac{1}{H_0}\nu - (x-z)\right\|_{L^2(\Ga)}
    \end{equation}
    and
    \begin{equation}\label{eq:estimate on term 2 in generalized 1.2}
        \begin{split}
            \left\|\frac{1}{H_0}\nu - (x-z)\right\|_{L^2(\Ga)}&\leq \left\|\frac{1}{H_0}-u_{\nu}\right\|_{L^2(\Ga)} +  \|\nabla h_z\|_{L^2(\Ga)}.
        \end{split}
    \end{equation}
    Combining \eqref{eq:elliptic regularity u Schauder}, Lemma \ref{lem: lemma 2 in appendix for C 1 alpha},\eqref{eq:Holder on fundamental identity}, \eqref{eq:estimate on term 1 in generalized 1.2} and \eqref{eq:estimate on term 2 in generalized 1.2} we conclude that 
    \begin{equation*}
        \left\|\frac{1}{H_0}\nu - (x-z)\right\|_{L^2(\Ga)} + \left\||x-z|-\frac{1}{H_0}\right\|_{L^2(\Ga)} \leq C \|H-H_0\|_{L^r(\Ga)}^{1/2},
    \end{equation*}
    where $C$ is a positive constant dependent on $N$, $\alpha$, $r$, $H_0$, $|\Ga|$ and the $C^{1,\alpha}$ regularity of $\Ga$; the dependence on $|\Ga|$ and $H_0$ can be removed and replaced with the dependence on $|\Om|$, by leveraging Lemma \ref{lem:bound on surface measure Ga}; in turn, the dependence on $|\Om|$ can be replaced by the dependence on the diameter of $\Om$.
    
    This concludes the proof of the result under the assumption that $\Ga$ is of class $C^{2}$. For $\Ga$ of class $C^{1,\alpha}\cap W^{2,r}$, we simply leverage the approximation result in \cite{Antonini}.
    \end{proof}
    
	\section{Proof of Lemma~\ref{lem:bound of flat W^2,2 Sobolev norm}}\label{Appendix: Proof of Lemma bound of Sobolev norms}
    \begin{proof}
    We start by describing the (standard) stereographic projection. Let $P=(0,...,1)\in\SS^{N-1}$, the stereographic projection from $P$ is given by $\iota:\SS^{N-1}\setminus\{P\}\rightarrow \RR^{N-1}$ be given by 
        \begin{equation}\label{eq:stereographic projection from the north pole}
            \iota(x_1,...,x_N)=\frac{1}{1-x_N}(x_1,...,x_{N-1}).
        \end{equation}
	
	The (standard) metric on $\SS^{N-1}$, $G(x)=(g_{ij}(x))$ in the local coordinates induced by $\iota$ is given by (see \cite[Chapter 1]{JJost})
	\begin{equation}\label{eq:metric on sphere}
		g_{ij}(x)= \frac{4}{(|x|^2+1)^2}\delta_{ij}, \quad x\in \mathbb{R}^{N-1}, 
	\end{equation}
	and $G^{-1}(x)=(g^{ij}(x))$ is given by 
	\begin{equation}\label{eq:inverse of metric on sphere}
		g^{ij}(x)= \frac{(|x|^2+1)^2}{4}\delta_{ij}, \quad x\in \mathbb{R}^{N-1}.
	\end{equation}
	
		By definition, we have (see \cite[Chapter 2]{Hebey})
		\begin{equation}\label{eq: W 2,p norm sphere}
			\|v\|_{W^{2,p}(\SS^{N-1})}= \left(\int_{\SS^{N-1}}|v|^p+|\nabla_{\SS^{N-1}}v|^p+ |D_{\SS^{N-1}}^2v|^p \,dS_x\right)^{1/2},
		\end{equation}
		where $\nabla_{\SS^{N-1}}v$ and $D_{\SS^{N-1}}^2v$ denote the first and second derivatives on the sphere. Since $v\circ \iota^{-1}$ denotes the function $v$ written in the coordinates induced by the stereographic projection, we will use the slight abuse of notation and also denote by $v$ the function $v\circ \iota^{-1}$.
        
        In the coordinates induced by the stereographic projection, we have 
		\begin{equation}\label{eq:first covariant derivative}
			|\nabla_{\SS^{N-1}}v(x)|^2 = g^{ij}(x)\partial_i v(x) \partial_j v(x) = \frac{(|x|^2+1)^2}{4}|\nabla v|^2, 
		\end{equation}
		\begin{equation}\label{eq:second covariant derivative}
			\begin{split}
				\left|D_{\SS^{N-1}}^2v\right|^2&= g^{ia}(x) g^{jb}(x)\left(\partial_{ij}v(x)-\Gamma_{ij}^k \partial_k v(x)\right)\left(\partial_{ab}v(x)-\Gamma_{ab}^k \partial_k v(x)\right)\\
				&= \frac{(|x|^2+1)^4}{16} \left(\partial_{ij}v(x)-\Gamma_{ij}^k \partial_k v(x)\right)^2\\
			\end{split}
		\end{equation}
		where we employed the Einstein summation convention, and where $\Gamma_{ij}^k$ denote the Christoffel symbols, which are given by 
	\begin{equation*}
		\Gamma_{ij}^k=\frac{1}{2}g^{nk}\left(\partial_j g_{in}(x)+\partial_i g_{jn}(x)-\partial_n g_{ij}(x)\right).
	\end{equation*}
	Using~\eqref{eq:metric on sphere} and~\eqref{eq:inverse of metric on sphere} we see that 
	\begin{equation*}
		\Gamma_{ij}^k= -\frac{2}{|x|^2+1}\left(x_j\delta_{ik}+x_i\delta_{jk}-x_k\delta_{ij}\right),
	\end{equation*}
	which, in particular, yields 
	$$
	\left|\Gamma_{ij}^{k}\right|\leq \frac{2|x|}{|x|^2+1},\quad \text{for all }x\in\mathbb{R}^{N-1}.
	$$
    Combining this, \eqref{eq:first covariant derivative} and \eqref{eq:second covariant derivative}, we see that
	\begin{equation*}
		\begin{split}
			\left|D_{\SS^{N-1}}^2v\right|^2&= \frac{(|x|^2+1)^4}{16} \left(\partial_{ij}v(x)-\Gamma_{ij}^k \partial_k v(x)\right)^2\\
			&= \frac{(|x|^2+1)^4}{16}\left((\partial_{ij}v(x))^2-2\partial_{ij}v(x)\Gamma_{ij}^k \partial_k v(x) +  \left(\Gamma_{ij}^k \partial_k v(x)\right)^2\right)\\
			&\geq \frac{(|x|^2+1)^4}{16}\left((\partial_{ij}v(x))^2\left(1-\left|\Gamma_{ij}^k\right|\right)+  \left(\left(\Gamma_{ij}^k\right)^2-|\Gamma_{ij}^k|\right) \left(\partial_k v(x)\right)^2\right)\\
			&\geq \frac{(|x|^2+1)^4}{16}\left((\partial_{ij}v(x))^2\left(1-\frac{2(N-1)|x|}{|x|^2+1}\right)-\frac{2(N-1)^2|x|}{|x|^2+1}\left(\partial_k v(x)\right)^2\right)\\
			&= \frac{(|x|^2+1)^4}{16}\left(1-\frac{2(N-1)|x|}{|x|^2+1}\right)|D^2 v|^2-\frac{1}{2}(N-1)^2|x|(|x|^2+1)\left|\nabla_{\SS^{N-1}} v\right|^2.
		\end{split}
	\end{equation*}
	Letting $R= \frac{1}{8(N-1)^2}$, we see that, for $|x|\leq R$ we have that 
	\begin{equation}\label{eq:pointwise estimate on derivatives}
		\begin{split}
		\left|D_{\SS^{N-1}}^2 v(x)\right|^2+\left|\nabla_{\SS^{N-1}}v(x)\right|^2&\geq  \frac{(|x|^2+1)^4}{32}|D^2 v|^2+\frac{1}{2}\left|\nabla_{\SS^{N-1}} v\right|^2\\
		&= \frac{(|x|^2+1)^4}{32}|D^2 v|^2+\frac{(|x|^2+1)^2}{8}|\nabla v|^2.
		\end{split}
	\end{equation}
        Using this and recalling \eqref{eq: W 2,p norm sphere}, we obtain the following chain
	\begin{equation*}
		\begin{split}
			\left\|v\right\|_{W^{2,p}(B_{R}^{N-1}(0))}^p&=\int_{B_{R}^{N-1}(0)} \left(|v(x)|^p+ |\nabla v(x)|^p+|D^2 v(x)|^p\right)\,dx\\
                &\leq \int_{B_{R}^{N-1}(0)} \left(|v(x)|^p + 2\left(|\nabla v(x)|^2+|D^2 v(x)|^2\right)^{p/2}\right)\,dx\\
			&\leq \int_{B_{R}^{N-1}(0)}\left(|v(x)|^p+2\frac{32^{p/2}}{(|x|^2+1)^p}\left(\frac{(|x|^2+1)^4}{32}|D^2 v|^2+\frac{(|x|^2+1)^2}{8}|\nabla v|^2\right)^{p/2}\right)\,dx\\
			&\leq 2^{1+\frac{5p}{2}}\int_{B_{R}^{N-1}(0)}\left(|v|^p+ \left(|\nabla_{\SS^{N-1}}v|^2+|D_{\SS^{N-1}}^2 v|^2\right)^{p/2}\right)dx\\
                &\leq 2^{N+5+\frac{5p}{2}}\int_{B_{R}^{N-1}(0)}\left(|v|^2+\left(|\nabla_{\SS^{N-1}}v|+|D_{\SS^{N-1}}^2 v|\right)^p\right)\frac{(|x|^2+1)^{N-1}}{2^{N-1}}dx\\
			&\leq 2^{N+5+\frac{5p}{2}}\int_{\SS^{N-1}}\left(|v|^p+\left(|\nabla_{\SS^{N-1}}v|+|D_{\SS^{N-1}}^2 v|\right)^{p}\right)d\mathcal{H}^{N-1}\\
			&=2^{N+5+\frac{5p}{2}}\left(\nr v\nr_{L^p(\SS^{N-1})}^p + \nr |\nabla_{\SS^{N-1}}v| + |D_{\SS^{N-1}}^2 v|\nr_{L^p(\SS^{N-1})}^p \right),
		\end{split}
	\end{equation*}
    where, in the first inequality, we used the pointwise inequality $a^p+b^p\leq 2(a+b)^p$ with $a,b\geq 0$, in the second inequality, we used~\eqref{eq:pointwise estimate on derivatives}, in the fourth inequality, we used the fact that $(a^2+b^2)^{p/2}\leq (a+b)^p$ with $a,b\geq 0$.
       From this, we now conclude the result using the triangle inequality
       \begin{equation*}
           \nr v\nr_{W^{2,p}\left(B_{R}^{N-1}(0)\right)}\leq 2^{\frac{N+5}{p}+ \frac{5}{2}}\left(\nr v\nr_{L^p(\SS^{N-1})}^p + \nr |\nabla_{\SS^{N-1}}v| + |D_{\SS^{N-1}}^2 v|\nr_{L^p(\SS^{N-1})}^p\right)^{1/p}\leq C \nr v\nr_{W^{2,p}(\SS^{N-1})},
       \end{equation*}
       where $C= 3 \cdot 2^{\frac{N+5}{p}+ \frac{5}{2}}$.   
    \end{proof}

	\section*{Acknowledgements}
	This work has been supported by the Australian Research Council (ARC) Discovery Early Career Researcher Award (DECRA) DE230100954 ``Partial Differential Equations: geometric aspects and applications''.
	João Gonçalves da Silva is supported by a Scholarship for International Research Fees at The University of Western Australia.
	Giorgio Poggesi is a member of the Australian Mathematical Society (AustMS) and was also supported by the Australian Academy of Science J G Russell Award.\\
	This work was carried out while João Gonçalves da Silva was visiting the University of Adelaide, which he wishes to thank for the hospitality and support.
    
    The authors also want to thank Enrico Valdinoci and Edoardo Proietti Lippi for very interesting discussions regarding Gagliardo-Nirenberg-type interpolation inequalities.
	
	\newpage

\end{document}